\documentclass[pdflatex,sn-mathphys-num]{sn-jnl}

\usepackage{graphicx}%
\usepackage{multirow}%
\usepackage{amsmath,amssymb,amsfonts}%
\DeclareMathOperator{\arcsinh}{arcsinh}
\usepackage{amsthm}%
\usepackage{mathrsfs}%
\usepackage[title]{appendix}%
\usepackage{xcolor}%
\usepackage{textcomp}%
\usepackage{manyfoot}%
\usepackage{booktabs}%
\usepackage{algorithm}%
\usepackage{algorithmicx}%
\usepackage{algpseudocode}%
\usepackage{listings}%
\usepackage{matlab-prettifier}
\usepackage{lscape}
\usepackage{tikz}

\newtheorem{theorem}{Theorem}
\newtheorem{remark}{Remark}%
\newtheorem{lemma}{Lemma}
\newtheorem{corollary}{Corollary}
\newtheorem{assumption}{Assumption}
\usetikzlibrary{patterns}
\pgfdeclarepatternformonly{spaced north east lines}
{\pgfpointorigin}{\pgfpoint{10pt}{10pt}}
{\pgfpoint{10pt}{10pt}}
{
  \pgfsetlinewidth{0.4pt}
  \pgfpathmoveto{\pgfpoint{0pt}{0pt}}
  \pgfpathlineto{\pgfpoint{10pt}{10pt}}
  \pgfusepath{stroke}
}

\newtheorem{definition}{Definition}%

\begin{document}

\title[Article Title]{ A column generation approach to exact experimental design}


\author*[1]{\fnm{Selin} \sur{Ahipa\c{s}ao\u{g}lu}}\email{S.D.Ahipasaoglu@soton.ac.uk}

\author[1]{\fnm{Stefano} \sur{Cipolla}}\email{S.Cipolla@soton.ac.uk}

\author[2]{\fnm{Jacek} \sur{Gondzio}}\email{j.gondzio@ed.ac.uk}

\affil*[1]{\orgdiv{School of Mathematical Sciences}, \orgname{University of Southampton}, \orgaddress{\street{Highfield Campus}, \city{Southampton}, \postcode{SO17 1BJ}, \country{UK}}}

\affil*[2]{\orgdiv{School of Mathematics}, \orgname{The University of Edinburgh}, \orgaddress{\street{King's Buildings}, \city{Edinburgh}, \postcode{EH9 3FD}, \country{UK}}}


\abstract{
In this work, we address the exact D-optimal experimental design problem when the number of design vectors is large. First, we propose a customized column generation algorithm to solve the continuous relaxation of the problem. In the approach, each restricted master problem is constructed carefully so that the number of variables stays small and therefore the subproblem can be solved efficiently by a Primal-Dual Interior-Point-based Semidefinite Programming solver. The support of this solution provides a subset of design points to be used in a local search algorithm for the solution of the integer problem. We prove that a local search algorithm  restricted to points of this subset provides an exact design that is provably close to the exact D-optimal design. Our numerical experiments show that, for large-scale instances in which the number of regression points exceeds by far the number of experiments, our approach achieves superior performance compared to existing branch-and-bound-based algorithms in both computational efficiency and solution quality.}



\maketitle

\section{Introduction}\label{sec:intro}

The optimal experimental design problem is fundamental in robust statistics. The simplest version of the problem assumes a framework in which a researcher designs a set of experiments by selecting a finite number of regression vectors from a given candidate set with an objective function that maximizes the expected information gain from the experiments. The experiments can be repeated several times on selected points, while some points in the candidate set might be unused. The total number of experiments (counting repetitions) is determined before the design of the experiments. 

Optimal experimental design has been an integral and critical element of clinical research in which experiments are costly and statistical robustness is highly desired. In this application area, the candidate set is determined by fixing a few factor levels for each control such as the dose of the medication. Often, the results of the experiments are used to train simple supervised models such as polynomial or logistic regression. In these applications, the size of the problem (number of candidate points in the set, number of parameters to be estimated, and number of experiments to be conducted) is relatively small. Therefore, optimal experimental designs (under the D-optimality criterion for linear regression models) can be calculated in a short amount of time, for example, by modeling the problem as a mixed-integer nonlinear problem, which in turn can be solved by existing off-the-shelf solvers (e.g. \cite{MR3396983}), customized branch-and-bound algorithms (e.g., \cite{MR4307385}), or using \textit{more sophisticated} mixed-integer programming techniques (e.g., \cite{MR4774635}). 

Despite these advances, certain instances of the problem remain challenging to solve exactly. These typically arise in machine learning applications and differ from the classical setting in terms of the ratios of the number of parameters to be estimated, the cardinality of the candidate set, and the number of experiments to be conducted.  One such example is subsampling: the selection of smaller subsets of data points from a large dataset as part of the model training process. Applying optimal design principles in the selection of subsamples, referred to as optimal subsampling (see \citep{yao_wang_2021} and references therein for a review), improves the robustness of the learning model while simultaneously decreasing the computational time spent on the training phase. Extracting a small set of relevant points,  tailored for the regression model, also allows for a focused analysis, reduces noise, and minimizes spurious correlations compared to random sampling. 

In this paper, we propose an algorithm to solve large-scale instances of the exact optimal experimental design problem where the number of candidate points is very large, while the number of parameters of the regression model is moderate. We demonstrate that the problem is particularly challenging when the number of experiments is close to the number of parameters to be estimated, in other words, when an exact design with a relatively small support is desired. This is typically the case for optimal subsampling and other machine learning applications.

{The most important feature of our work is the control of the size of the problems that need to be solved in each iteration of the algorithm. Each problem remains small enough to be tackled by second-order methods in a short amount of time, allowing high accuracy solutions even for large scale problems.}

\subsection{Problem formulation} \label{sec:prob_formulation}


Given a finite number of regression points $x_1, . . . , x_m \in \mathbb{R}^n$ that span  
$\mathbb{R}^n$,  
an \emph{exact experimental design (EED) of size $N$} can be defined as a set of
non-negative integers $n_1, \dots, n_m$ such that $\sum_{i=1}^mn_i=N$. For each $i \in \{1,2,\dots,m\}$, the integer $n_i$ corresponds to the number of experiments to be carried out with the vector $x_i$ as the input and $N$ is the total number of experiments. 
The \textit{support of an experimental design} is the set of regression points on which at least one experiment is carried out, that is, $x_i$ is in support if $n_i>0$.  
	

Under the assumption that errors in experiments are identically distributed random variables with mean zero and standard deviation $\sigma$ and a generalized linear model is used in regression, the Fisher information matrix of an EED is independent of the parameters to be estimated and is equal to $\frac{N}{\sigma^2}\sum_{i=1}^m \frac{n_i}{N} x_i x_i^T$. An \emph{exact D-optimal experimental design} maximizes the determinant of this matrix and can be calculated by solving the following nonlinear integer program \citep{MR4307385}:
\begin{equation} \label{eq:ED} 
	\begin{aligned}
		\max &  \; g_0(u):= \ln \det(XUX^T) \\
		 s.t.   &   \; \sum_{i=1}^m u_i = 1,  \\
		         & \; u_i \geq 0   & \hbox{ for } i =1, \dots, m, \\
		         &  \;  u_i N \in \mathbb{Z} & \hbox{ for } i =1, \dots, m, \\
	\end{aligned}
\end{equation}
where  $[u_i]_{i=1}^m := [\frac{n_i}{N}]_{i=1}^m \in \mathbb{R}^m$, $ X:=\begin{bmatrix} 
	x_1 | \cdots | x_m
\end{bmatrix} $, and $U:=diag(u)$.

    In the following, we will refer to the continuous relaxation of problem \eqref{eq:ED} as \emph{the limit problem}. This is also known as \emph{the approximate D-optimal experimental design problem} in the literature. We prefer the first nomenclature, as the continuous version of the problem can be viewed as the limit of problem \eqref{eq:ED} when the number of experiments $N$ goes to infinity. The limit problem is a convex program (note that the objective function to be maximized is a concave function and the feasible region is the unit simplex) and can be solved by any of the known methods for constrained convex programming. Solving the limit problem efficiently is important  in tackling problem \eqref{eq:ED} regardless of the global optimization techniques used for the integer version. We will provide a brief literature review for the limit problem in Section \ref{sec:lit_rev} and present an efficient method for large-scale instances where $m \gg n$ in Section \ref{sec:ColumnGen}.

We have previously mentioned that when problem \eqref{eq:ED} is tackled directly as a mixed-integer nonlinear program, computational limitations are encountered relatively quickly. This is demonstrated in Figure \ref{fig:boscia_study}. This figure presents the boxplots relative to the computational time required to solve $10$ randomly generated instances of \eqref{eq:ED} with  $n=30$, and for each of $m \in \{100,\dots 500\}$ and $N \in \{n, n+5, n+10 \}$ using the method proposed in \cite{MR4774635} (and relative software). To the best of our knowledge, this is one of the state-of-the-art of mixed-integer solvers for the exact D-optimal experimental design problem. The results clearly indicate that instances with $N=n$ are significantly more challenging to solve and that the computational time quickly increases {with $m$}. 
Moreover, the figure clearly shows how the solver performance appears to be highly sensitive to the specific instance considered, although this variability tends to decrease as $N$ increases. This suggests that the problem is particularly challenging when $N$, the number of experiments,  is close to $n$, the number of parameters to be estimated. The current state-of-the-art might not be sufficient to address these instances.

\begin{figure}[h] 
\includegraphics[width=\textwidth]{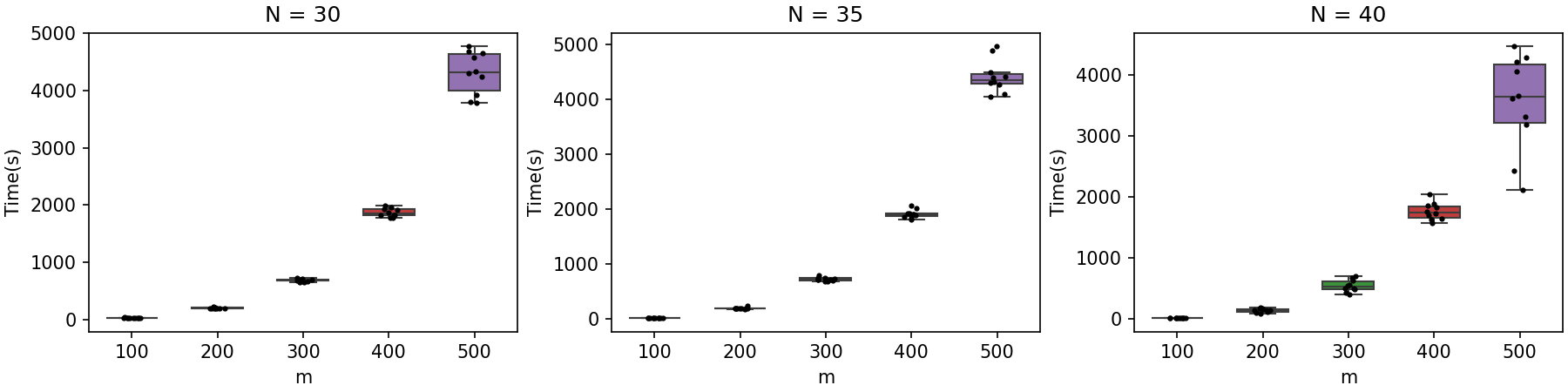}
\caption{Computational times for random instances generated and solved using the software from \cite{MR4774635}.}
\label{fig:boscia_study}
\end{figure}

In this work, we address the limitations of mixed-integer solvers in the regime  $N\approx n$ by proposing a hybrid approach. To be precise, we will consider $n,m,N$ s.t. the following Assumption \ref{ass:m,n} holds.
\begin{assumption} \label{ass:m,n}
$n$ is in general a small integer ($n\leq 50$), $m \gg n$, e.g., $ m \in O(n^c)$ for $c \in \mathbb{N}$, and $N \in \{n, \dots,  n+10\}$.
\end{assumption}

\section{Literature review and contribution}\label{sec:lit_rev}

It is known that the \emph{minimum-volume enclosing ellipsoid} problem is dual to the limit problem and strong duality holds. The duality of these two problems leads to an interesting geometric interpretation of the optimality conditions for both problems. In particular, when a regression vector is in the support of the D-optimal experimental design, it must also be on the surface of the minimum-volume ellipsoid that encloses all the regression vectors in the candidate set.  An early but excellent investigation of the duality of the two problems is provided in \cite{MR2224698}, while \cite{MR2376769} provides a modern view in which the geometric interpretation is extended to approximate optimal solutions. {In particular, when an approximate solution to the limit problem is obtained, an approximate optimal solution to the dual problem can be constructed. The dual solution corresponds to an ellipsoid that does not necessarily enclose all data points. Nevertheless, by scaling it up with an appropriate constant, one can obtain a feasible ellipsoid that encloses all data points. Similarly, scaling it down with another constant excludes points that would be in the support of the optimal solution of the limit problem. 
}

The limit problem and its dual can be solved by a variety of numerical methods to any desired accuracy. {Since the objective function is not defined on the extreme points and some faces of the feasible region, these methods need to be initiated carefully at a point in the domain of the objective function and the subsequent iterates must stay within the domain.} The performance of the algorithms changes significantly with the size and properties of the dataset. {Moreover, it is also important to note that most of those employing merely first-order information  are highly sensitive to the initial solution (See comparison of initialization methods in \cite{MR2158428, MR2376769}.)}

Most of the widely used and analyzed first-order methods for the limit problem are variants of the Frank-Wolfe (FW) algorithm. {One of the best performing variants is the Frank-Wolfe algorithm with away steps \cite{guelat1986away} with initial solution chosen according to the algorithm in \cite{MR2158428}. In this, the algorithm is started at a solution with small cardinality. The next iterate is chosen by either increasing the weight of the solution at a regression point that is not covered by the dual ellipsoid (an ordinary Frank-Wolfe step) or decreasing the weight of one of the points that is in the current support (Wolfe's away step). In both cases, the FW method can be viewed as a coordinate descent algorithm. The optimal step size is calculated in closed-form by solving a one-dimensional maximization problem  where the objective function corresponds to the linearization of the original objective function at the current iterate and constraints make sure that the next iterate stays in the intersection of the unit simplex and the domain of the objective function.  The objective function value and the gradient at the next iterate are calculated efficiently due to the fact that the Fisher information matrix at the new iterate is a rank-one update over the current one. (Details of this MVEE-tailored FW variant can be found in \cite{MR2376769}.)} 
These algorithms are terminated when a primal-dual pair with a small optimality gap is obtained and have desirable convergence properties especially when the algorithm is started with an iterate with a small support following the initialization strategy suggested in \cite{MR2158428} and Wolfe's away steps are used as in \cite{MR2348357,MR2376769}. In particular, \cite{MR2376769,MR2158428} show that a $\delta$-approximate solution to the limit problem can be obtained in $\mathcal{O}(n \ln n + \frac{n}{\delta})$ steps. This is also the size of the support of the output of the algorithm. Furthermore, this global convergence result, which holds for any starting point {in the domain of the objective function}, provides a pessimistic view {as the FW method with away steps performs significantly better in practice than this bound suggests}. It was proven in \citep{MR2376769} that there exist data-dependent constants $P$ and $Q$ such that the number of steps is only  $\mathcal{O}(P+ Q {\ln (\delta^{-1})})$. I.e., this method is locally linearly convergent and the number of iterations does not depend on the number of data points in the neighborhood of the optimal set. {The global linear convergence of the algorithm has recently been established in \cite{Zhao2025AwayStepFW}, using a key observation that the objective function is a logarithmically-homogeneous self-concordant barrier. This is an interesting development as the known global linear convergence property of the method for smooth objective functions (e.g., \cite{lacoste-julien2015global}) doesn't apply to the D-optimal design problem whose objective function is not smooth at all points of the feasible region.} Furthermore, the gradient of the objective function evaluated at a feasible solution provides critical information on identifying regression vectors that cannot be in the support of an optimal solution following the discussion in \cite{MR2339022}. Including regular checkpoints to eliminate such points from the dataset greatly improves the efficiency of first-order methods and makes solving large-scale instances of the problem possible, as demonstrated in \cite{ThesisAhi}. This is important since each iteration of the algorithm takes $\mathcal{O}(mn)$ operations where $m$ is the number of regression vectors in the candidate set and reducing its size leads to significant gains in computational time.
To tackle the challenges posed by ``tall" datasets - those too large to fit into working memory - \cite{MR4387269} introduced the Big Index Batching (BIB) algorithm, which iteratively optimizes over a small subset of points, uses the elimination rule  from \cite{MR2339022} to delete points, and then adapts the subset by adding new points from the larger dataset. The BIB algorithm emphasises minimizing data reads and is coupled with a first-order method to find the optimal solution for each batch. Although the authors do not mention this, the BIB algorithm can be viewed as a column generation or active-set method as discussed in Section  \ref{sec:ColumnGen}.

Alternatively, the limit problem can also be tackled using higher-order methods, e.g., \cite{MR2091768} uses the dual reduced Newton method together with active set strategies, which works well for moderate problem sizes. In addition, \cite{MR3396983} formulates the problem as a second-order cone program and utilizes MOSEK, and in \cite{MR2061575} an SDP formulation is given which can be solved directly using CVX. The choice between first-order and higher-order methods remains a relevant research question. A comparative analysis of several first and second methods is provided in \cite{MR4655115} together with a classification of the difficulty of the data sets using kurtosis as the relevant metric. Their work also includes an implementation of an active-set method, drawing inspiration from the general strategy outlined in \cite{MR462607} - a seminal paper published in 1978. The availability of their Python codes facilitates further research and benchmarking.
While second-order methods demonstrate potential, as highlighted by \cite{MR4655115,MR3396983}, their scalability often presents a significant bottleneck for very large-scale problems. Similarly, solving the limit problem using semidefinite programming (SDPs) or second-order cone programming (SOCPs) techniques, is only possible when the problem size is small due to their high computational complexity.
Here, we provide an algorithm that uses column generation and elimination methods to reduce the size of the problem to a scale that is tractable by SDP solvers. In particular, we solve the smaller subproblems using a barrier method for the Linear Matrix Inequality Representation of the problem. This enables us to use a higher-order method for the subproblems, while using first-order information to generate the subproblems in each step of the column generation algorithm.

Although the problem formulation (\ref{eq:ED}) is an integer program that is known to be NP-hard, the use of greedy methods for its solution has been widespread with good performance in practice.
Early approaches are mostly variants of the two-exchange heuristic that iteratively swaps pairs of points until a 2-optimal solution is reached. A well-established version is provided in \cite{MR2323647}, with even earlier origins in Federov's work \cite{MR403103}.  While computationally efficient, these heuristics do not guarantee global optimality. Also, when they are applied to the full dataset, they take considerable amount of time for large scale instances of the problem. The first attempts at exact solutions include a simple branch-and-bound algorithm provided by \cite{MR653110}, nevertheless, this method can handle only very small instances of the problem.
    
Modern research into exact D-optimal design can broadly be categorized into three main lines: (i) customized or generic branch-and-bound (BaB) implementations, (ii) methods based on quadratic programming (QP) approximations, and (iii) Mixed-Integer Second-Order Cone Programming (MISOCP) reformulations that leverage off-the-shelf solvers. There appears to be a growing recent interest in this area, potentially driven by advancements in solver technology and algorithmic frameworks for mixed-integer nonlinear programming. 

A customized BaB algorithm is provided by \cite{MR4307385}, where subproblems correspond to bounded versions of the limit problem 
and its dual. The subproblems are solved using a customized Frank-Wolfe (FW) algorithm with Wolfe's away steps, where the algorithm is initiated using solutions of the parent node, and the bounds on the variables are respected in each iteration of the FW algorithm. More recently, \cite{MR4774635} have used Boscia.jl, a contemporary algorithmic framework to solve nonlinear integer programs using BaB. This framework also employs FW algorithms for solving node relaxations to approximate optimality, sharing similarities with the approach {introduced in \cite{MR4307385}} but potentially benefiting from more efficient solver components or step-size rules. Similarly, \cite{ponte2023branchandbounddoptimalityfastlocal}  presented a BaB algorithm that features variable tightening techniques derived from the dual of bounded subproblems, similar to those in {\cite{MR4307385,MR4774635}}. Additionally, the BaB given in \cite{ponte2023branchandbounddoptimalityfastlocal} incorporates spectral and Hadamard bounds with promising results. The relevance of these as cutting planes for related problems such as the Minimum Volume Enclosing Ellipsoid (MVEE) problem could be further investigated. They also employ several local search methods, including a 2-exchange algorithm and a randomized SVD-based approach, to improve lower bounds within their Julia-based implementation, which solves subproblems with Knitro within the Juniper BaB solver, tackling problems with around 100 candidate points.

In the domain of QP-based approximations, \cite{MR4102952} proposed a method that constructs a quadratic approximation of the D-optimality objective function in the neighborhood of an optimal solution of the limit problem. This leads to a quadratic integer program (IQP) which is solved using Gurobi. Later, \cite{MR4102952} introduced the AQuA (Ascent with Quadratic Assistance) algorithm, which employs IQP or MIQP-based approximations. While AQuA can handle a large number of candidate points in low-dimensional settings, it lacks theoretical guarantees and has reportedly shown variable performance.

MISOCP reformulations offer another avenue for exact solutions. The first MISOCP formulation for the exact D-optimal design was provided in \cite{MR3396983}, which was solved utilizing CPLEX via PICOS. Their work highlighted a graph-theoretic interpretation related to maximum spanning trees and noted that the inclusion of linear constraints significantly accelerates the solution process. They provided computational comparisons against the 2-exchange heuristic mentioned above on block-design and chemical kinetics datasets. 

The collective efforts indicate a drive towards scalable and provably optimal methods, leveraging both problem-specific insights and general advancements in mathematical optimization solvers and frameworks. However, very large instances of the problem are beyond the reach of any of the exact methods discussed above. Therefore, development of approximation algorithms with provable optimality guarantees is a recent area of interest. In this regard, \cite{madan2019combinatorial} has shown that the 2-exchange algorithm and other greedy heuristics are
asymptotically optimal when $N\gg n$, which is different from the assumptions we have in this work. 

It is also worth mentioning related work on other optimality criteria. For example, \cite{MR4307385} studies the BaB algorithm for Kiefer's optimality criteria in general and provides the necessary details for the exact A-optimal design, where the objective function is $(trace(XUX^T)^{-1})^{-1}$. The MISOCP formulations of \cite{MR3396983} also include A-optimality. In addition, \cite{MR4759553} recently formulates the exact A-optimal design problem as a Mixed-Integer Linear Program (MILP) with McCormick relaxations, demonstrating the broader applicability of mixed-integer programming to optimal design problems beyond D-optimality.


\subsection{Contribution and organization}

Our contribution focuses on the efficient computation of exact D-optimal designs in settings where the dimension $n$ is moderate, but the number of candidate design points $m$ is much larger. In such a regime, {a method that requires partial or full} enumeration, e.g., a Branch-and-Bound algorithm, is computationally infeasible. Broadly speaking, exploiting the fact that the cardinality of the support of D-optimal designs remains moderate independently from the number of data points, i.e., exploiting the sparsity of the continuous relaxation, our method combines rapid support identification - implemented via a column generation strategy integrated with an Interior-Point-based SDP solver - with a local search algorithm restricted to the identified support. On the one hand, this approach is theoretically justified, as it produces solutions with the same quality guarantees for the objective function as applying the local search algorithm to the entire dataset. On the other hand, our numerical experiments demonstrate that the framework proposed here can compute high-quality approximate solutions for the exact D-optimal design at scales never achieved before. More in detail, our contribution is organized as follows:

\begin{itemize}
    
\item  In Section \ref{sec:ColumnGen}, we frame and reinterpret the approach proposed in \cite{MR4387269} for computing the minimum-volume ellipsoids in the context of Column Generation. In our formulation, the master problem maintains a restricted set of design points, while the pricing subproblem identifies the most informative point to be added. 
This perspective elucidates the algorithmic framework and related convergence properties presented in \cite{MR4387269}, anchoring the method within the established Linear Programming literature while still acknowledging and recognizing the unique challenges and peculiarities brought by the MVEE problem. Moreover, to solve the restricted master problems at each iteration, we propose and justify the use of an Interior-Point-based method, namely SDPT3 \cite{MR1976479}, for computing high-accuracy solutions. This feature enables the rapid identification of the optimal support. {It is interesting to note that our proposal extends the applicability of primal-dual IPM based SDP solvers for the computation of minimum volume ellipsoids for datasets with millions of data points.} 

\item In Section \ref{HPconstant}, we seek provable approximation guarantees for local search algorithms when the pairs to be considered for exchange are restricted to the support of the solution of the continuous relaxation.  In particular, we prove that a local search approach for the computation of exact D-optimal design solutions, when applied solely to the identified support set, achieves the same error bounds as when applied to the entire dataset. This significantly reduces computational cost while preserving theoretical guarantees on the worst-case estimate.

\item In Section \ref{sec:numerics}, we demonstrate the practical advantages of our approach, proposing a broad spectrum of numerical results. The considered large-scale \textit{Synthetic and Real World} datasets are suitably modified to represent challenging instances of the exact D-optimal design problem. And indeed, the results presented there confirm the efficiency and robustness of our approach, showcasing how our proposed computational framework is able to consistently outperform state-of-the-art first-order methods \cite{MR2376769, MR3522166} for the solution of the MVEE and state-of-the-art mixed-integer based solvers \cite{MR4307385, MR4774635} for the exact D-optimal design problem in terms of computational time and quality of computed solutions. 

{Notably, as already mentioned for the MVEE setting,  we also provide experimental evidence demonstrating the computational limitations of monolithic application of IPM-based solvers and how the Column Generation framework proposed in this paper addresses these by decomposing the problem into smaller restricted master problems: maintaining SDPT3's efficiency while achieving scalability to datasets with millions of candidate points.}

In the exact D-optimal design setting, the  numerical results {presented in Section \ref{sec:numerics}}
show that our method reliably produces solutions whose objective values match those from a mixed-integer programming approach, while running orders of magnitude faster, thereby enabling high-quality exact D-optimal designs at scales previously considered unattainable, cf. Figures  \ref{fig:boscia_study} and \ref{fig:colgen_study}. { We acknowledge that the computational times in such comparison should be interpreted with appropriate caveats: Boscia is an exact solver that provides globally optimal solutions, whereas our proposal is a heuristic method with worst-case approximation guarantees.}

\begin{figure}[h] 
\includegraphics[width=\textwidth]{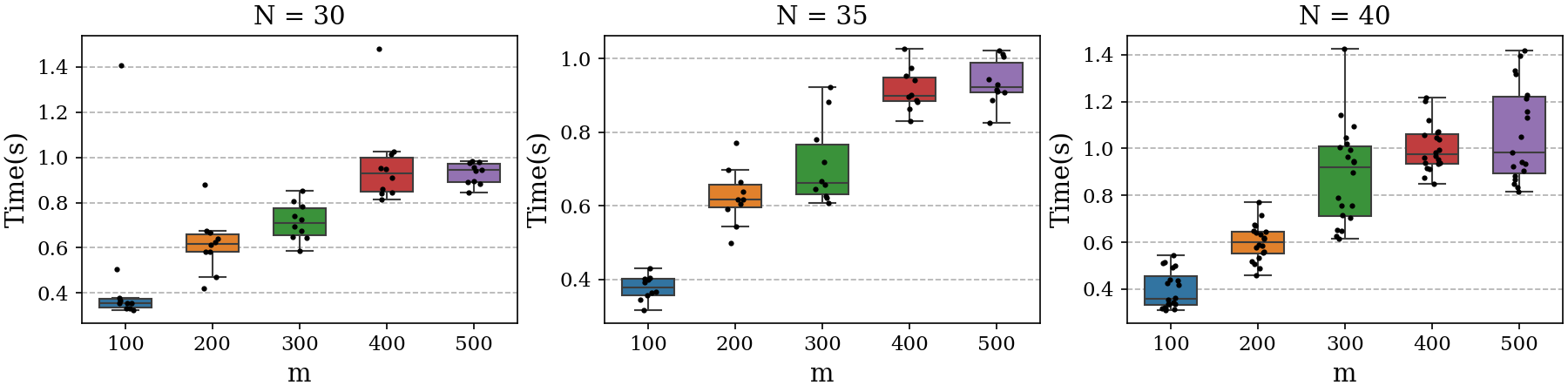}
\caption{Computational times for the same random instances of  Figure  \ref{fig:boscia_study} solved using our proposal.}
\label{fig:colgen_study}
\end{figure}
\end{itemize}


\section{Column generation} \label{sec:ColumnGen}
In this section, always under Assumption \ref{ass:m,n}, we neglect the integrality constraints $u_iN \in \mathbb{Z}$ and study possible approaches to solve the limit problem. We start by briefly reviewing the related duality theory.
\subsection{Duality}
 We refer the interested reader to \cite[Ch. 2]{MR3522166} from which this session is borrowed. Let us consider the following convex problem:

\begin{equation} \label{eq:ED_dual_1} \tag{DMP}
	\begin{aligned}
		\min_{H \succ 0} &  \; f_0(H):=    \ln \det H^{-1}  \\
		s.t.   &   \; x_i^THx_i \leq n,  & \hbox{ for } i =1, \dots, m,\\
	\end{aligned}
\end{equation}
which is known as the Minimum Volume Enclosing Ellipsoid problem (MVEE) in the literature. For $u \geq 0$, consider the Lagrangian function:

\begin{equation}\label{eq:Lag}
	L_0(H, u):=   \ln \det H^{-1} +\sum_{i=1}^m u_i(x_i^THx_i-n),
\end{equation}
where $u_i$ corresponds to the multiplier of the $i^{th}$ constraint. Using the Lagrangian in \eqref{eq:Lag}, the KKT conditions of problem \eqref{eq:ED_dual_1} can be stated as

	\begin{align}
		x_i^THx_i & \leq n,  & \hbox{ for } i =1, \dots, m,   \label{KKT_1}   \tag{PF} \\
		u_i  & \geq 0, & \hbox{ for } i =1, \dots, m,   \label{KKT_2} \tag{DF} \\
		u_i (x_i^THx_i -n) & = 0, & \hbox{ for } i =1, \dots, m,   \label{KKT_3} \tag{Comp}  \\
		\nabla_HL_0(H,u)= -H^{-1}+XUX^T&= 0.  \label{KKT_4} \tag{Stat}
	\end{align}
\begin{lemma}\label{lem:min+lag}
	We have that
	\begin{equation}\label{eq:min_lagrangian}
		\min_{H} L_0(H,u)=\ln \det(XUX^T) +n(1-\sum_{i=1}^mu_i).
	\end{equation}
\end{lemma}

\begin{proof}
	
	We have
	
	\begin{equation*}
		 -\ln\det H=\ln\det H^{-1}
	\end{equation*}
	and hence,  using \eqref{KKT_4}, i.e., $XUX^T=H^{-1}$, we have that $$-  \ln\det H =  \ln \det (XUX^T).$$
	Moreover, always using \eqref{KKT_4} in second part of \eqref{eq:Lag}, we have ${\rm{trace}}(HXUX^T)=n$, which proves hence \eqref{eq:min_lagrangian}. 
\end{proof}

Using Lemma \ref{lem:min+lag}, we have hence that the dual problem of \eqref{eq:ED_dual_1}  is exactly the continuous relaxation of \eqref{eq:ED} (limit problem), i.e.,

\begin{equation} \label{eq:ED_no_integer_constraints}  \tag{MP}
	\begin{aligned}
		\max &  \; g_0(u)   \\
		s.t.   &   \; \sum_{i=1}^m u_i = 1,  \\
		& \; u_i \geq 0   & \hbox{ for } i =1, \dots, m. \\
	\end{aligned}
\end{equation}

\begin{remark}
It is important to note that, in the derivation above, for the sake of simplicity, we derived \eqref{eq:ED_no_integer_constraints} as dual of \eqref{eq:ED_dual_1}.
On the other hand, in the following, we refer to and interpret \eqref{eq:ED_no_integer_constraints} as the {Primal} Master Problem and  \eqref{eq:ED_dual_1} as the Dual Master Problem. This is fully justified by the fact that strong duality holds, see also Theorem \ref{th:existence_uniqueness} below.
\end{remark}

\begin{remark} For any primal dual feasible points $(u,H)$, weak duality follows also observing that
	\begin{equation*} \label{eq:duality_gap} 
		\begin{split}
			  &  -\ln\det H-\ln\det(XUX^T) = \\
			&  -\ln \det ( H XUX^T) = \\
            &  - \ln (\prod_{i=1}^{n} \lambda_i) \geq - n \ln ( \sum_{i=1}^{n} \lambda_i /n ) = - n \ln (n/n)=0,
		\end{split}
	\end{equation*}
 where we used the arithmetic-geometric mean inequality and the fact that  $\sum_{i=1}^n\lambda_i={\rm{trace}} (HXUX^T)=n$.  
\end{remark}

We are now ready to state the main theorem of this section that clarifies the existence and uniqueness of solutions of \eqref{eq:ED_no_integer_constraints}-\eqref{eq:ED_dual_1}, see \cite[Th. 2.2]{MR3522166}.

\begin{theorem} \label{th:existence_uniqueness}
If $X$ has full rank, then \eqref{eq:ED_dual_1} has a unique optimal solution $H^*$, \eqref{eq:ED_no_integer_constraints} has an 
optimal solution $u^*$ with $XU^*X^T$, and $f_0(H^*) = g_0(u^*)$.
\end{theorem}

\vspace{0.2cm}

\begin{definition}
 Throughout this work,  $({u}^*, {{H}}^*)$  will denote a primal-dual solution of~\eqref{eq:ED_no_integer_constraints}~-~\eqref{eq:ED_dual_1}, i.e., it will satisfy \eqref{KKT_1}-\eqref{KKT_4}.
\end{definition}

\subsection{Column generation for the limit problem}
We start the description of the column generation approach for solving the limit problem and its dual by introducing the \textit{Restricted Master Primal-Dual problems.} The Primal Restricted Master Problem (RMP) is obtained as the restriction of \eqref{eq:ED_no_integer_constraints} to a subset $\bar{M}\subset \{1,\dots,m\}=:M$, i.e.,

\begin{equation} \label{eq:ED_Primal_Restricted} \tag{RMP}
	\begin{aligned}
		\max &  \; \ln \det(X\bar{U}X^T)  \\
		s.t.   &   \; \sum_{i \in \bar{M} } \bar{u}_i = 1,  \\
		& \; \bar{u}_i \geq 0   & \hbox{ for } i  \in \bar{M}. \\
	\end{aligned}
\end{equation}
It is important to note that $\bar{u} \in \mathbb{R}^{|\bar{M}|}$ and, hence, that the dimension and complexity of the solution of \eqref{eq:ED_Primal_Restricted} depends on $|\bar{M}|$. For this reason, to fix ideas, we can assume $|\bar{M}| \in O(n)$ and $|\bar{M}| \geq (n+1)$. Using the  discussion carried out in the previous section, we can see that the dual of \eqref{eq:ED_Primal_Restricted} has again the form

\begin{equation} \label{eq:ED_dual_Restricted}  \tag{DRMP}
	\begin{aligned}
		\min &  \; -  \ln \det(\bar{H})  \\
		s.t.   &   \; x_i^T\bar{H}x_i \leq   n,  & \hbox{ for } i  \in \bar{M}. \\
	\end{aligned}
\end{equation}
We will use, what is common in the column generation scheme, 
a \textit{pricing} procedure \cite{MR2980569,MR2193875} to identify all (or a subset of) violated dual constraints. Such constraints will be appended to the set $\bar{M}$. Let us assume that primal-dual solution $(\bar{u}^*,\bar{{H}}^*)$ of \eqref{eq:ED_Primal_Restricted}-\eqref{eq:ED_dual_Restricted} has been computed, and find 
%

\begin{equation}\label{eq:entering_support}
	z(\bar{H}^*):= \max_{i \in M} \{ 0, x_i^T\bar{H}^*x_i-n \} .
\end{equation}
Following on the discussion on the geometric interpretation of optimality conditions above, 
$\bar{{H}}^*$ determines the current approximation of the MVEE which contains all points in the index set $\bar{M}$, while the points $x_i$ such that $x_i^T\bar{H}^*x_i-n>0$ lie outside this ellipsoid, and therefore, correspond to constraints of (DMP) which are violated. Moreover, $z(\bar{H}^*)>0$ implies the existence of an index $\bar{i}$ such that the point $x_{\bar{i}}$ lies outside of the ellipsoid defined by $\bar{H}^*$.  We also observe that every optimal solution  $\bar{u}^*$ of the problem \eqref{eq:ED_Primal_Restricted} corresponds to a feasible solution $u$ of problem \eqref{eq:ED_no_integer_constraints}, where $u_i=\bar{u}^*_i$ for $i \in \bar{M}$, and $u_i=0$ for $i \in M \setminus \bar M$. Hence 
$$\ln \det(\bar{X}\bar{U}^*\bar{X}^T) = \ln det(XUX^T)  \leq \ln \det(X{U}^*X^T),$$ 
where $\bar{X}$ is the submatrix of $X$ built of columns $i \in \bar M$ only. 
Moreover, since $\bar{u}_i(n-x_i^T\bar{H}^*x_i  ) =0$ for all $i \in \bar{M}$, using the definition above, we have $u_i(n-x_i^T\bar{H}^*x_i  ) =0$ for all $i \in M$. Hence $\bar{H}^*$  is an optimal solution of (MP) iff ~$z(\bar{H}^*)~=~0$.

Using the ideas introduced until now, we are finally ready to present the full details of a column generation approach  for the solution of problems \eqref{eq:ED_no_integer_constraints}-\eqref{eq:ED_dual_1}, see Algorithm \ref{alg:column_generation_Pure}.  Such a framework generalizes the column generation approach usually used in  linear programming, see, e.g.  \cite{MR2980569}, to the nonlinear problem in \eqref{eq:ED_no_integer_constraints}.

\begin{algorithm}
	\caption{Column Generation for \eqref{eq:ED_no_integer_constraints} }
	\label{alg:column_generation_Pure}
	\begin{algorithmic}[1]
		\Procedure{ColumnGeneration}{}
		\State Initialize $u$ using the technique proposed in  \cite{MR2158428};
		\State Initialize $\bar{M}:=\{ i \in M \hbox{ s.t. }  u_i>0 \} $;
        \State {Solve \eqref{eq:ED_Primal_Restricted}-\eqref{eq:ED_dual_Restricted} restricted to  $\bar{M} $ to produce initial $(\bar{u}^*,\bar{{H}}^*)$;}
		\While{$ z(\bar{H}^*) :=	\max_{i \in M} \{ 0, x_i^T\bar{H}^*x_i-n \}  > 0$};
		\State {Find}  $\bar {i}$ s.t. $ \bar{i} \; { \in } \; \arg\max_{i \in M} x_{i}^T\bar{H}^*x_{i}-n$; \label{algline:max_violation}
        \State Define  $\bar{M} =  \bar{M} \cup \{ \bar{i}\}$; \label{algline:support}
		\State Solve \eqref{eq:ED_Primal_Restricted}-\eqref{eq:ED_dual_Restricted} restricted to  $\bar{M} $ to produce a new $(\bar{u}^*,\bar{{H}}^*)$; \label{algline:IPM}
		\EndWhile
        \State For $i \in \bar{M}$ set $u^*_i := \bar{u}_i^*$ and $u^*_i:=0$ otherwise. Also $H^*:=\bar{H}^*$.
		\EndProcedure
	\end{algorithmic}
\end{algorithm}
The convergence of Algorithm \ref{alg:column_generation_Pure} is analyzed in  Theorem \ref{th:conv_pure}.

\begin{theorem} \label{th:conv_pure}
	Algorithm \ref{alg:column_generation_Pure} converges to a Primal-Dual optimal solution in at most $m$ iterations.
\end{theorem}

\begin{proof}
	Let us define $g^*$ as the optimal value of \eqref{eq:ED_no_integer_constraints}. Let us denote with $(\bar{u}^*,\bar{{H}}^*)$ the solution of  \eqref{eq:ED_Primal_Restricted}-\eqref{eq:ED_dual_Restricted} obtained at Line \eqref{algline:IPM} of Algorithm  \ref{alg:column_generation_Pure}. Two situations may occur:
	\begin{itemize}
		\item  If $z(\bar{H}^*) > 0$, then there exists  $  \bar{i}  \in \{1, \dots, m\}  \hbox{ s.t. }  x_{\bar{i}}^T \bar{H}^*x_{\bar{i}} -n >0  $.  Since $\bar{i}$ is added to $\bar{M}$, the corresponding dual constraints will not be violated in the next iterations. Therefore, it guarantees the progress of the algorithm. Also, this case can only happen at most $m$ times.
        \item  If $z(\bar{H}^*) = 0 $, then Algorithm \ref{alg:column_generation_Pure} stops. In this case, consider $(\bar{u}^*,\bar{{H}}^*)$ to be a solution of the Primal-Dual pair \eqref{eq:ED_Primal_Restricted}-\eqref{eq:ED_dual_Restricted}. An optimal solution $({u}^*,{{H}}^*)$ of  \eqref{eq:ED_no_integer_constraints}-\eqref{eq:ED_dual_1} is obtained by defining ${u_{i}^*}=\bar{u}_i^*$ for $i \in \bar{M}$ and  ${u_{i}^*}=0$ for $i \in M \setminus \bar{M}$ and ${{H}}^*=\bar{{H}}^*$.
	\end{itemize}
 \end{proof}

\begin{remark}
 A similar proof to the one proposed for Theorem \ref{th:conv_pure} holds when $n_0>1$ points violating the dual constraints are added to the set $\bar{M}$. In that case, Algorithm \ref{alg:column_generation_Pure} converges in at most $\lceil m/n_0 \rceil $ iterations. Moreover, it is easy to see that the convergence proof still holds if at Line~\ref{algline:max_violation} of Algorithm \ref{alg:column_generation_Pure} the $\arg\max$ is substituted by any $\bar {i}$ s.t. $ x_{i}^T\bar{H}^*x_{i}-n>0$, i.e., any point that violates the dual-constraints.
\end{remark}

\vspace{0.1cm}

It is important to note, at this stage, that since the cardinality of $\bar{M}$ increases at every step of Algorithm \ref{alg:column_generation_Pure}, the dimension of the problems \eqref{eq:ED_Primal_Restricted}-\eqref{eq:ED_dual_Restricted}  to be solved, increases at every step. This issue represents one of the main computational bottlenecks and,  in the next section, we will show how it can be overcome for the particular case of D-optimal design thanks to the use of the Harman-Pronzato constant \cite{MR2339022} that identifies which points in set $M$ can be completely dropped from the data set since these cannot be in the support of the optimal solution of (MP).




\subsection{On the use of Harman-Pronzato constant \cite{MR2339022}}
\label{HPconstant} 

Let us define the constant

$$h_n( \varepsilon ):= n\left(1+\frac{\varepsilon}{2}  - \frac{\sqrt{\varepsilon(4+\varepsilon -4 /n ) } } {2} \right), $$
where $\varepsilon := \max_{i  = 1,\dots, m} x_i^THx_i-n $ and $H=(XUX^T)^{-1}$ for any given primal feasible $u \geq 0$. In \cite[Th. 2]{MR2339022}, it is proven that a point $x_{j}$ that satisfies $x_j^THx_j<h_n( \varepsilon )$ 
can not be a support point of a D-optimal design, i.e., s.t. $u_j>0$. {Observe that the quantity $\varepsilon$ is a function of the matrix $H$, and thus must be recomputed for each dual solution $H$. For notational simplicity, however, we suppress this dependence and write $\varepsilon$ instead of $\varepsilon(H)$.}

\begin{definition}
    In the following, given $H$, we will say that a point $x$ satisfies the Harman-Pronzato condition, in short HP-condition, if $x^THx<h_n( \varepsilon )$.
\end{definition}

\begin{remark} \label{rema:Pronzato}
Using that $n \geq h_n(\varepsilon)$ for all $\varepsilon>0$, see \cite{MR2339022}, we have that
\begin{equation*}
	x^THx -n \leq x^THx -h_n(\varepsilon).
\end{equation*}
 Hence 
 \begin{equation*}
     \begin{split}
       &  x^THx -n > 0 \Rightarrow  x^THx -h_n(\varepsilon) > 0,  \\
       & x^THx -h_n(\varepsilon) < 0 \Rightarrow x^THx -n  < 0.
     \end{split}
 \end{equation*}
The above implications show that all points $x$ that are not covered by the current estimated MVEE defined by $H$, i.e., $x^THx -n >0$ will fail the condition. Whereas all points that satisfy the HP-condition, i.e., $x^THx<h_n( \varepsilon )$, lie indeed in the interior of the currently estimated MVEE. At the same time, the implications above also identify a set of points such that $ x^THx -n  < 0 \leq   x^THx -h_n(\varepsilon)$. These are of particular relevance because they lie in the interior of an estimated ellipsoid $H$ but are close to its boundary and potentially might become active in the MVEE defined by the optimal solution $H^*$, see Figure \ref{fig:pictorial-Pronzato}.  

{Note that when $\varepsilon = 0$, this is equivalent to the geometric interpretation of the optimality conditions and points in the interior of the optimal ellipsoid satisfy the HP-condition. }

\begin{figure}[ht]
\centering
\begin{tikzpicture}[scale=0.90]
	
	\draw[thick, red] (0,0) ellipse (3cm and 2cm);  
	\draw[thick, blue, fill=white] (0,0) ellipse (2cm and 1cm);  
	
    \draw (-1.8,-1.7) node[left] {{$x^T H x > h_n(\varepsilon)$}}; 
	\draw  (0,0.5) node[below] {{$x^T H x < h_n(\varepsilon)$}};
	\draw  (0,2) node[below] {\textcolor{red}{$x^T H x <n$}};
	\draw  (0.6,-0.91)  node[below] {{$x^T H x = h_n(\varepsilon)$}}; ;
    \draw  (2.9,0.5)  node[right] {\textcolor{red}{$x^T H x = n$}}; 
    \draw  (2.8,-1) node[right] {\textcolor{red}{$x^T H x > n$}};

    \begin{scope}[transparency group, opacity=0.1]
		\fill[blue, even odd rule] 
			(-5,-4) rectangle (6,4)  
			(0,0) ellipse (2cm and 1cm);  
	\end{scope}

    \begin{scope}[transparency group, opacity=0.3]
		\fill[pattern=spaced north east lines, pattern color=red, even odd rule] 
			(-5,-4) rectangle (6,4)  
			(0,0) ellipse (3cm and 2cm);  
	\end{scope}

\end{tikzpicture}
\label{fig:pictorial-Pronzato}
\caption{ {Pictorial representation of the regions identified by the HP-condition (see Remark \ref{rema:Pronzato}). The figure illustrates three distinct regions in the design space based on the quadratic form $x^T H x$. The inner region (white), bounded by the blue ellipse $x^T H x = h_n(\varepsilon)$, represents designs where $x^T H x < h_n(\varepsilon)$. The light blue shading outside the blue ellipse corresponds to designs satisfying $h_n(\varepsilon) < x^T H x$. The exterior region (red diagonal lines) represents designs where $x^T H x > n$. These regions partition the design space according to $H$ with $h_n(\varepsilon)$ serving as a threshold parameter that depends on  $\varepsilon := \max_{i  = 1,\dots, m} x_i^THx_i-n $. }}
\end{figure}
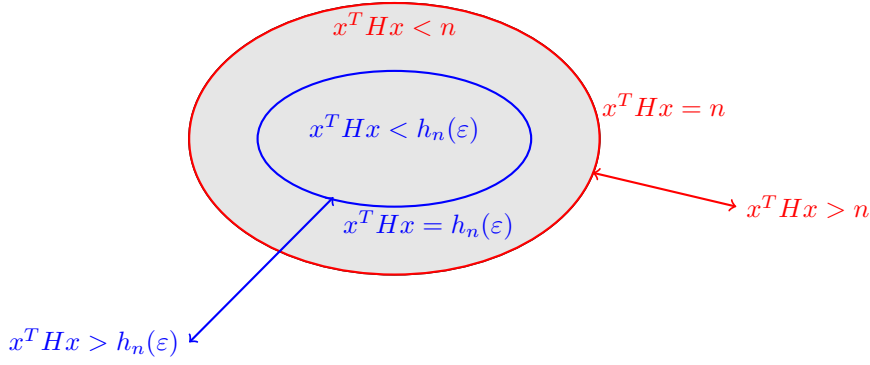

\end{remark}

In light of Remark \ref{rema:Pronzato}, it  is natural to modify Line  \ref{algline:support} of Algorithm \ref{alg:column_generation_Pure} using the information provided by the HP-condition, i.e., substituting such line with 

\begin{equation*}
	\bar{M} =  (\bar{M} \cup \{ \bar{i}\}) \cap  \bar{J},
\end{equation*}	
where $\bar{J} :=\{ i \in M \hbox{ s.t. }  x_i^T\bar{H}^*x_i \geq h_n(\varepsilon) \} $. Indeed, points s.t. $ x^THx -h_n(\varepsilon) < 0$ can be safely discarded from $\bar{M}$, reducing hence the dimension of the \eqref{eq:ED_Primal_Restricted}-\eqref{eq:ED_dual_Restricted} to be solved. On the other hand, as per the discussion in Remark \ref{rema:Pronzato}, $\bar{M}$  defined as above could still contain points s.t. $x^THx -n< 0$ (especially in the initial phases of the algorithm). This might still represent a computational issue/bottleneck and could lead to sets $\bar{M}$ having unnecessarily high number of points.  To prevent this and guarantee an efficient solution of problems \eqref{eq:ED_no_integer_constraints}-\eqref{eq:ED_dual_1},  we consider the following Algorithm~\ref{alg:column_generation}, a reformulation of the Big Index Batching Algorithm, see \cite[Alg. 1]{MR2339022}, where the dimensions of the \eqref{eq:ED_no_integer_constraints}-\eqref{eq:ED_dual_1} to be solved -- the batch size -- is not upper bounded. In this framework, the points which are detected as non-support for the currently estimated MVEE defined by $\bar{H}^*=(X\bar{U^*}X^T)^{-1}$, i.e., points $x_{i}$ such that $\bar{u}_i^*=0$, are eliminated from $\bar{M}$ reducing hence the dimension of the problem \eqref{eq:ED_Primal_Restricted}-\eqref{eq:ED_dual_Restricted} to be solved in the next iteration.  In particular, given  $\bar{M}$ and the corresponding $(\bar{u}^*,\bar{{H}}^*)$, we define the next $\bar{M}$ as follows:

\begin{enumerate}
	\item Define  $\bar{A}:=\{ i \in \bar{M} \hbox{ s.t. }  \bar{u}^*_i >0  \};$
    \item Define the new $\bar{M}$ as
	$$\bar{M}= (\bar{I} \cup  \bar{A}), $$
	where  $\bar{I}$ is the set of  the $n_0$ most violated dual constraints (we assume that there are at least $n_0$ of such constraints). 
\end{enumerate}


Moreover, defining $\bar{J} :=\{ i \in M \hbox{ s.t. }  x_i^T\bar{H}^*x_i \geq h_n(\varepsilon) \} $, we know that  
the points in $\bar{J}^C$ can be safely eliminated from the dataset, see Remark \ref{rema:Pronzato}. Using Remark \ref{rema:Pronzato} again, we note that at Line  \ref{algline:infeasibility} of Algorithm \ref{alg:column_generation}, it holds
$$ x_{i}^T\bar{H}^*x_{i}-n >0 \Rightarrow  x_{i}^T\bar{H}^*x_{i}-h_n(z(\bar{H}^*))>0,$$ 
which implies that the indexes $i$ corresponding to violated dual constraints  belong indeed to $\bar{J}$.

\begin{algorithm}
	\caption{Column Generation for \eqref{eq:ED_no_integer_constraints} {with HP elimination}}
	\label{alg:column_generation}
	\begin{algorithmic}[1]
			\Procedure{ColumnGeneration}{}
			\State  Initialise $M_P=M$
			\State Choose $n_0 \in \mathbb{N}$;
			\State Choose an initial $\bar{M} \subset M_P$ using  \cite{MR2158428};
			\State Initialize $z(\bar{H}^*)=+\infty$;
			\While{$ z(\bar{H}^*)    > 0$} \label{algline:stopping_criterion}
			\State Solve \eqref{eq:ED_Primal_Restricted}-\eqref{eq:ED_dual_Restricted} restricted to  $\bar{M} $ to produce $(\bar{u}^*,\bar{{H}}^*)$ \label{algline:Primal_Dual_Sol}
			 \State Define $\bar{A}:=\{ i \in \bar{M} \hbox{ s.t. }  \bar{u}^*_i >0  \} $ \label{algline:supports}
			\State Define $\bar{I}:=\{ i \in M_P \hbox{ s.t. }  n_0 \hbox{ largest elements } x_{i}^T\bar{H}^*x_{i}-n>0  \}$ \label{algline:infeasibility}
			\State Define $\bar{M}= (\bar{I} \cup  \bar{A}) $
			\State 	Compute $z(\bar{H}^*)=\max_{i \in M_P} \{ 0, x_i^T\bar{H}^*x_i-n \}$
			\State Define $h_n(z(\bar{H}^*))$   as in \cite[Th. 2]{MR2339022} \label{algline:Pronzato_constant} 
			\State Define $\bar{J} :=\{ i \in M_P \hbox{ s.t. }  x_i^T\bar{H}^*x_i \geq h_n(z(\bar{H}^*)) \} $ 
			\State Define $M_P:=M_P \setminus \bar{J}^C$ \label{algline:Pronzato_removal} 
			\EndWhile
			\State For $i \in \bar{M}$ set $u^*_i := \bar{u}_i^*$ and $u^*_i:=0$ otherwise. Also $H^*:=\bar{H}^*$.
			\EndProcedure
		\end{algorithmic}
\end{algorithm}
We are finally able to prove the finite convergence of Algorithm \ref{alg:column_generation}.  The proof of finite convergence is partially borrowed from \cite{MR2339022}. 

\begin{theorem}
Algorithm \ref{alg:column_generation} terminates in a finite number of iterations.
\end{theorem}
\begin{proof}
	The first important observation is that  Line  \ref{algline:Pronzato_removal} of Algorithm \ref{alg:column_generation} does not discard any point of the support at optimality, hence, it can be ignored. We consider then the version of Algorithm \ref{alg:column_generation} where $M_P$ is always equal to $M$ and where the lines connected to the use of the Harman-Pronzato constant are ignored (Lines \ref{algline:Pronzato_constant} -\ref{algline:Pronzato_removal}). 
	To show that Algorithm \ref{alg:column_generation} terminates in a finite number of iterations, let us argue by contradiction. Suppose that Algorithm \ref{alg:column_generation} as just described,  does not terminate in a finite number of steps. Then there exists an infinite  sequence $\{ \bar{A}_i \}_{i \in \mathbb{N}}$ of computed supports (see Line \ref{algline:supports}). Let us denote by $\{ \bar{g}_i\}_{i \in \mathbb{N}}$ the corresponding function values of the function $g_0(u)$. As at every iteration we are adding at least one point outside the current estimated ellipsoid, see Line \ref{algline:infeasibility},  we have that $\bar{g}_{i+1}>\bar{g}_{i}$, i.e., $\bar{g}_i$ is a strictly increasing sequence. The contradiction follows from observing that there exists only a finite number of possible subsets of the given dataset. 
\end{proof}

{
As a final remark in this section, we note that at Line~\ref{algline:Primal_Dual_Sol} of Algorithm~\ref{alg:column_generation}, we did not specify the particular solver employed for addressing the restricted problems~\eqref{eq:ED_Primal_Restricted}–\eqref{eq:ED_dual_Restricted}. In our numerical experiments, we adopt a Primal-Dual Interior-Point Method with a stringent accuracy tolerance (see Section~\ref{sec:solution RMP} for further details). However, it is worth emphasizing that the pricing step remains valid even if these restricted problems are solved inexactly. Similarly, the HP criterion used at Line~\ref{algline:Pronzato_constant} does not require any kind of optimality, as the constant $h_n(\varepsilon)$ is well-defined for any primal feasible point of~\eqref{eq:ED_no_integer_constraints}.
This observation suggests that one may consider using approximate solutions to reduce the computational footprint associated with solving~\eqref{eq:ED_Primal_Restricted}–\eqref{eq:ED_dual_Restricted}. While this can improve the overall efficiency of the method, it may come at the cost of a slower HP-based column elimination, as inexact solutions can lead to less aggressive identification of the points that can not be on the support of the MVEE.
}
{
Furthermore, in practice, the stopping condition at Line~\ref{algline:stopping_criterion} is modified to $z(\bar{H}^*) > n\delta$. And indeed, when the condition $z(\bar{H}^*) \leq n\delta$ is satisfied, 
the output of the algorithm $({u}^*,{H}^*)$ satisfies the {following $\delta$-primal feasibility condition of \cite{MR2376769}} for \eqref{eq:ED_no_integer_constraints}-\eqref{eq:ED_dual_1}:
	\begin{align*}
		x_i^THx_i & \leq (1+\delta) n,  & \hbox{ for } i =1, \dots, m,   \\
		u_i  & \geq 0, & \hbox{ for } i =1, \dots, m,    \\
		XUX^T&= H^{-1}.  
	\end{align*}
{This implies that the optimality gap is at most $n\delta$, see Lemma 2.1, \cite{MR2376769}.}

\section{Local search algorithms for exact D-optimal design} \label{sec:local_search}
In this section, we present the local search approach considered to produce an approximate integer solution exploiting a computed primal-dual solution $(u^*,H^*)$ of \eqref{eq:ED_no_integer_constraints}-\eqref{eq:ED_dual_1}, i.e., the limit problem.
Let us define the support of $u^*$ as $$S:=\{ i \in 1,\dots,m \, : {u_{i}^*}>0  \}$$ and $\phi_{Rel}:=g_0(u^*)=f_0(H^*)$. By definition of support $S$ and using the fact that $(u^*,H^*)$ is the primal-dual solution of \eqref{eq:ED_no_integer_constraints}-\eqref{eq:ED_dual_1}, the optimal objective function values of the following two problems are equal 

\begin{equation} \label{eq:primal_dual_on_support}
    \begin{aligned}
		\max &  \; g_0(u) \\
		s.t.   &   \; \sum_{i \in S} u_i = 1,  \\
		& \; u_i \geq 0   \;\; \hbox{ for } i \in S, \\
	\end{aligned} \quad \quad {\mbox { and }} \quad \quad  \begin{aligned}
		\min &  \; f_0(H) \\
		s.t.   &   \;  x_i^THx_i \leq   n,  & \hbox{ for } i  \in S ,\\ \\
	\end{aligned}
\end{equation}
and are, in turn, equal to $g_0(u^*)=f_0(H^*)$. Problem \eqref{eq:primal_dual_on_support} coincides, indeed, with \eqref{eq:ED_dual_1}-\eqref{eq:ED_no_integer_constraints} where the original set $M$ has been substituted by the support $S$.

We will exploit the equality in \eqref{eq:primal_dual_on_support}  to apply a local search algorithm for the solution of the exact D-optimal design problem, where the original set $M$ is substituted by the possibly  \textit{much smaller} set $S$ in order to produce an approximate integer solution.

Before we present and analyze the local search algorithm considered here, we recall that there may exist a vector $\hat{u}^* \neq u^*$ such that
\[
\Bigl(\sum_i \hat{u}^*_i\,x_i x_i^T\Bigr)^{-1} = H^*,
\]
i.e., the support $S$ of the minimizers does not have to be uniquely determined. 
In fact, different algorithms are very likely to find different 
optimal supports. For example, if the limit problem is solved using the FW algorithm without Wolfe's away steps initialized with a uniform distribution, i.e., with all identical $u_i$, then the algorithm will fail in identifying an optimal solution with sparse support in finitely many steps. In contrast, if the algorithm is initialised with a sparse vector and Wolfe’s away steps are employed, then the resulting solution may indeed exhibit a sparse support. 
Indeed, in general, the cardinality of the support depends on the problem dimension, the used algorithmic framework and the target accuracy of the solution, while remaining independent of the total number of points. As discussed in Section \ref{sec:lit_rev}, \cite{MR2158428} shows that an optimal solution with a small support can be obtained with a careful initialization.

At first glance, this ambiguity might seem to be quite problematic when applying a local search approach to exploit the support of a computed primal solution $u^*$. A first clarifying observation in this regard is that, even though there might be multiple optimal solutions to the primal problem, they correspond to the same (unique) dual solution $H^*$.
Therefore, we can use any of such primal solutions to determine a support set $S$. Interestingly enough, exploiting the uniqueness of $H^*$, we will prove that the local exchange phase (which will be described next, see Algorithm \ref{alg:local-search-d-design}), produces an approximate integer solution having worst-case guarantees not depending on the particular primal optimal solution used -- and its corresponding support. 


\subsection{Bounds Based on a local search algorithm} \label{sec:bounds_local_search}
In this section, we show that, when applying the local search algorithm (see Algorithm~\ref{alg:local-search-d-design}) to the exact D-optimal design problem, it is sufficient to restrict the search to the support $S$ of an optimal solution to the limit problem. In particular, we show that the approximation guarantee obtained by running the algorithm on~$S$ is identical to that achieved when the algorithm is executed over the full data set~$M$. Given the combinatorial nature of the local search procedure, this reduction is expected to yield significant computational savings in cases where $|M| \gg |S|$.


 The proof of this result follows easily from the analysis in \cite{madan2019combinatorial}; for completeness, we briefly outline the main steps in the following. To this aim, consider the local search procedure described in Algorithm~\ref{alg:local-search-d-design}, which iteratively replaces elements in the current multi-subset $I$ with elements from the set $S$, provided such exchanges lead to an increase in the determinant of the current candidate solution (see Line~\ref{algline:exchange}). The use of multi-set notation permits repeated inclusion of indices $i \in S$, and we denote by $n_i$ the multiplicity of each index $i \in I$. The initialisation of $I$ is based on an optimal solution $u^*$ of the limit problem. Specifically, $I$ is constructed by including $n_i$ copies of each index $i \in S$, where $n_i$ is a suitable integer approximation of $N u_i$ satisfying the condition $\sum_{i \in S} n_i = N$.

\begin{algorithm}
\caption{Local search algorithm for $D$-optimal experimental design}
\label{alg:local-search-d-design}
\begin{algorithmic}[1]
\State \textbf{Input:} $S$,  $I$ any multi-subset of $S$ of size $N$ such that $
   G \;=\; \sum_{i \in I} x_i\,x_i^T$ is a non-singular matrix.

\While{there exist $i \in I$ and $j \in S$ such that 
\[
   \det\bigl(G - x_i\,x_i^T + x_j\,x_j^T\bigr) \;>\; \det(G)
\]} \label{algline:exchange}
    \State $G \gets G - x_i\,x_i^T + x_j\,x_j^T$
    \State $I \gets \bigl(I \setminus \{i\}\bigr) \cup \{j\}$
\EndWhile
\State \textbf{Output:} $(G,I)$
\end{algorithmic}
\end{algorithm}

\begin{definition}
    For any $
   G \;=\; \sum_{i \in I} x_i\,x_i^T$  and for all $i,j \in \{1,\dots,m\}$, let us define $$\tau_i:= x_{i}^TG^{-1} x_{i} \hbox{ and } \tau_{ij}:= x_{i}^TG^{-1} x_{j}.$$
\end{definition}

In the following Lemma \ref{lem:claim1}, we state an important property satisfied by $\tau_i$ and $\tau_{ij}$ defined above.

\vspace{0.1cm}

\begin{lemma} [see Claim 1 in \cite{madan2019combinatorial}] \label{lem:claim1}
    Let $(G,I)$ denote the output of Algorithm \ref{alg:local-search-d-design}. Then for all $i \in I$ and $j \in S$ it holds
\begin{equation*}
    \tau_j -\tau_i\tau_j +\tau_{ij}\tau_{ji} \leq \tau_{i}.
\end{equation*}
    
\end{lemma}

We are now ready to state the worst-case approximation bound for approximate integer solutions computed using Algorithm \ref{alg:local-search-d-design}.

\vspace{0.1cm}

\begin{lemma}
  Let us denote by $\phi_D:= \log \det(X{U_{E}^*}X^T)$ where ${u_{E}^*}$ is an optimal solution of {Exact D-Optimal problem} \eqref{eq:ED} and $(G,I)$ denote the output of Algorithm \ref{alg:local-search-d-design}. Then
  \begin{equation} \label{eq:worst_case_bound}
      \log \det(\frac{1}{N}G) \geq -h(N,n) + \phi_{Rel}  \geq -h(N,n) + \phi_D,
  \end{equation}
  where $h(N,n)$ is a constant that depends only on $n$ and $N$.
\end{lemma}

\begin{proof}
    Given $G$, the strategy of the proof is based on exhibiting a feasible solution of the dual problem in \eqref{eq:primal_dual_on_support} that will be used to obtain the desired result. Let us consider $Y= \frac{\alpha}{N} G$. We have $Y^{-1}=\frac{N}{\alpha}G^{-1}$ and suppose that we want to find a constant $\alpha$ s.t. $x_j^TY^{-1}x_j = \frac{N}{\alpha}x_j^TG^{-1}x_j \leq n$ for all $j \in S$. In that case, we would have indeed, 
    \begin{equation*}
        \phi_{Rel} \leq f_0(\frac{N}{\alpha}G^{-1})= - \log \det(\frac{N}{\alpha}G^{-1})=  \log \det(\frac{\alpha}{N}G) = n \log({\alpha}) +\log \det(\frac{1}{N}G).
    \end{equation*}
 Therefore, calculating an upper bound on $\tau_j$ would be useful when finding an $\alpha$ which satisfies the desired condition, i.e., we would like to bound $\max_{j \in S}x_j^TG^{-1}x_j$. Such a bound is provided in Lemma 9 of \cite{madan2019combinatorial}:
    
\begin{equation*}
    \tau_j \leq \frac{n}{N-n+1}.
\end{equation*}
{Since} we want $\alpha$ s.t. $\frac{N}{\alpha}\frac{n}{N-n+1} \leq n$, i.e., $\alpha \geq \frac{N}{N-n+1} \geq 1$, we can set $h(N,n):=n\log(\frac{N}{N-n+1}) $ and observe that $\phi_{Rel} \geq \phi_D$ {to complete the proof}.

\end{proof}

\begin{remark} \label{rem:difficulty}
    It is easy to see that if Algorithm \ref{alg:local-search-d-design} uses the whole dataset $M$ rather than $S$, as proved in \cite{madan2019combinatorial}, one can obtain exactly the same approximation bound as stated in \eqref{eq:worst_case_bound}.  
    Moreover, \eqref{eq:worst_case_bound} confirms the very well-known fact that the exact D-optimal design problem is more difficult to solve when $N \approx n$ and easier when $N \gg n$ as, in this case, the local search algorithm would indeed produce solutions having very good approximation bounds since $\lim_{N \to \infty } \frac{N}{N-n+1} = 1$.

\end{remark}

\vspace{0.5cm}

\begin{corollary} \label{cor:worst_case estimate}
If $(G,I)$ is the output of Algorithm \ref{alg:local-search-d-design} and ${u_{E}^*}$ is an optimal solution of \eqref{eq:ED}, then    
\begin{equation*}
\det(G) \geq (\frac{N-n+1}{N})^n \det(XN{U_{E}^*}X^T).    
\end{equation*}
\end{corollary}

\section{Numerical Results} \label{sec:numerics}
In this section, we present a series of numerical results to illustrate the efficiency and robustness of our proposal when compared to the state of the art methods. We will start showcasing the strength of our proposal in the fast identification of support for \textit{challenging datasets} for the MVEE problem, see Section \ref{sec:numerics_MVEE}, and then we will present numerical results concerning the computational performance of the approach proposed in this paper for the solution of large exact D-optimal design problems, see Section \ref{sec:exact-D_optimal-experiments}. Our implementation is written in \texttt{Matlab 2024b} and all the numerical results are obtained using \textit{Iridis X} equipped with \textit{2.35 GHz AMD 7452 Processor}. All the software is publicly available at \url{https://github.com/StefanoCipolla/D_Optimal_Design_Matlab}.  In the remainder of this section we discuss  the details about the datasets and our implementation. 

\subsubsection{Solution of the Restricted Problems and Stopping Criteria} \label{sec:solution RMP}
In Algorithm \ref{alg:column_generation} we use the following experimental setting. Given $\bar{M}$, every \eqref{eq:ED_Primal_Restricted}-\eqref{eq:ED_dual_Restricted} pair of problems is solved with SDPT3 \cite{MR1976479} using \texttt{OPTIONS.gaptol$=1e-9$}. Despite this being quite a strict tolerance, we noticed that such a parameter is mainly responsible for the fast identification of the support of the limit MVEE problem and largely justifies the use of a second-order method. Concerning the stopping criterion, we stop Algorithm \ref{alg:column_generation} when  $ z(\bar{H}^*) < toll$ with $toll = 1e-5$. Please note that this represents an \textit{absolute} stopping condition, i.e., not depending on the data of the problem, and is thus inherently stringent. 

\subsection{Datasets}  \label{sec:dataset}

In this work, we consider the following datasets:

\begin{itemize}
    \item \textbf{Synthetic}: The dataset is generated with the code used to produce the dataset in \cite{MR2376769}, where each instance is generated as a mixture of five Gaussians with random means and covariances, see \cite{MR2091768}. 
    We set $(n,m)\in \{ 10,20,30,40,50 \} \times \{100K, 1M, 10M\}$ and generate 10 instances for each couple of parameters $(n,m)$.  
    \item \textbf{Real World}:  We consider three large-scale instances from  UCI Dataset, see Table \ref{tab:problem_sizes} for the corresponding details.
    
\begin{table}[htbp]
  \centering
  \caption{Problem dimensions }
  \label{tab:problem_sizes}
  \begin{tabular}{lrr}
    \toprule
    Problem                 &     \(n\) &         \(m\) \\
    \midrule
    HIGGS               &       29   & 11\,000\,000  \\
    SUSY                &       19   &  5\,000\,000  \\
    SGEMM GPU    &       18   &    241\,600   \\
    \bottomrule
  \end{tabular}
\end{table}

\end{itemize}
Following the approach in \cite{MR4655115}, for any instance \(X \in \{\mathbf{Synthetic},\,\mathbf{Real\ World}\}\), we apply the $\sinh–\arcsinh$ transformation by setting

\[
X \longmapsto \sinh (\tfrac{1}{p}\,\arcsinh(X)).
\]
This transformation grants fine-grained control over the kurtosis of the dataset and, as shown in \cite[Sec.~3]{MR4655115}, directly affects the support size of the minimum volume enclosing ellipsoid (MVEE). And indeed, as also demonstrated in the following numerical results, empirically, the support size provides an effective measure of the problem’s difficulty. We consider $p \in \texttt{logspace}(0,2,5)$. Finally, it is important to note that in all the subsequent figures, the label \texttt{``Kurtosis''} will indicate the average ($\log$)kurtosis computed component-wise, see \cite[Sec. 3.1]{MR4655115}.

\subsection{The MVEE problem} \label{sec:numerics_MVEE}

The experimental results presented in this section are designed to demonstrate the computational advantages of our proposed methodology for solving Minimum Volume Enclosing Ellipsoid (MVEE) problems relative to existing state-of-the-art solvers. Specifically, our evaluation focuses on establishing the superior performance characteristics of the column generation approach, Algorithm \ref{alg:column_generation}, in terms of both computational efficiency and algorithmic robustness when identifying support sets for computationally challenging problem instances.

The comparative analysis benchmarks our proposed method against a MVEE-tailored Frank-Wolfe-type algorithm {with away steps presented in \cite{MR2376769} and discussed in detail in Section \ref{sec:lit_rev}. (We refer to this particular variant of the more general Frank-Wolfe algorithm with away steps simply as  the \textit{FW method} in this section.)} To ensure fair and meaningful comparisons, both algorithmic implementations were developed using \texttt{Matlab}, with the stopping tolerance for the FW method set to $toll = 1e-5/n$. This tolerance parameter was specifically chosen to guarantee comparable final duality gaps between our proposed approach and the baseline method.

The comprehensive numerical results are presented in Figures \ref{fig:fw-cg-comparison} and \ref{fig:fw-cg-comparison-RW}, which detail the performance evaluation across \textbf{Synthetic} datasets ($m=10M$) and \textbf{Real World} datasets, respectively. The experimental framework incorporates two distinct data configurations: the upper panels of both figures present results obtained using datasets modified through the $\sinh–\arcsinh$ transformation with increasing average kurtosis scores from left to right, while the lower panels display outcomes for the \textit{original}, unmodified datasets.

Examination of the computational time comparisons, as illustrated in the first row subplots of each figure, reveals that the methodology proposed in this work consistently and substantially outperforms the FW approach across the evaluated problem instances. Particularly noteworthy is the performance advantage observed for problems with dimension $n=50$, where our approach achieves computational speedups of approximately two orders of magnitude for a significant proportion of the test problems. This substantial performance improvement is most pronounced for problem instances characterized by large support sets, as evidenced by the analysis of the row subplots labeled \texttt{``Support''} in the upper panels of the aforementioned figures. This phenomenon is particularly prevalent in cases where the observed kurtosis of the specific problem instance is relatively small, which corroborates and extends the experimental observations reported in \cite[Sec.~3]{MR4655115}.

Beyond computational efficiency gains, our approach also demonstrates superior solution quality characteristics. Analysis of the row subplots labelled as \texttt{``error''} reveals that the final duality gap achieved by the proposed method is, in general, substantially smaller than that obtained by the FW approach, indicating enhanced convergence properties and solution accuracy.

Finally, we would like to draw the reader's attention to the row subplots labelled as \texttt{``support''}. As the results presented in Figures \ref{fig:fw-cg-comparison} and \ref{fig:fw-cg-comparison-RW} demonstrate, the cardinality of the computed support for ColGen is, typically, greater than the cardinality of the computed support with FW, which should, in general, be attributed to the different mechanism the two approaches use to identify such set, see also the discussion at the beginning of Section \ref{sec:local_search}.

%

Collectively,  the results presented in Figures \ref{fig:fw-cg-comparison} and \ref{fig:fw-cg-comparison-RW} provide compelling experimental  evidence that the proposed Column Generation strategy, coupled with an Interior Point Solver SDPT3 \cite{MR1976479}, significantly outperforms FW across multiple performance indicators. Our approach, when applied to the \textit{large-scale} problems considered in this section, demonstrates superior computational efficiency, enhanced algorithmic robustness, improved solution quality, and notably more consistent performances that are less dependent on specific instances.

\begin{figure}[htbp]
  \centering
  \includegraphics[width=\textwidth]{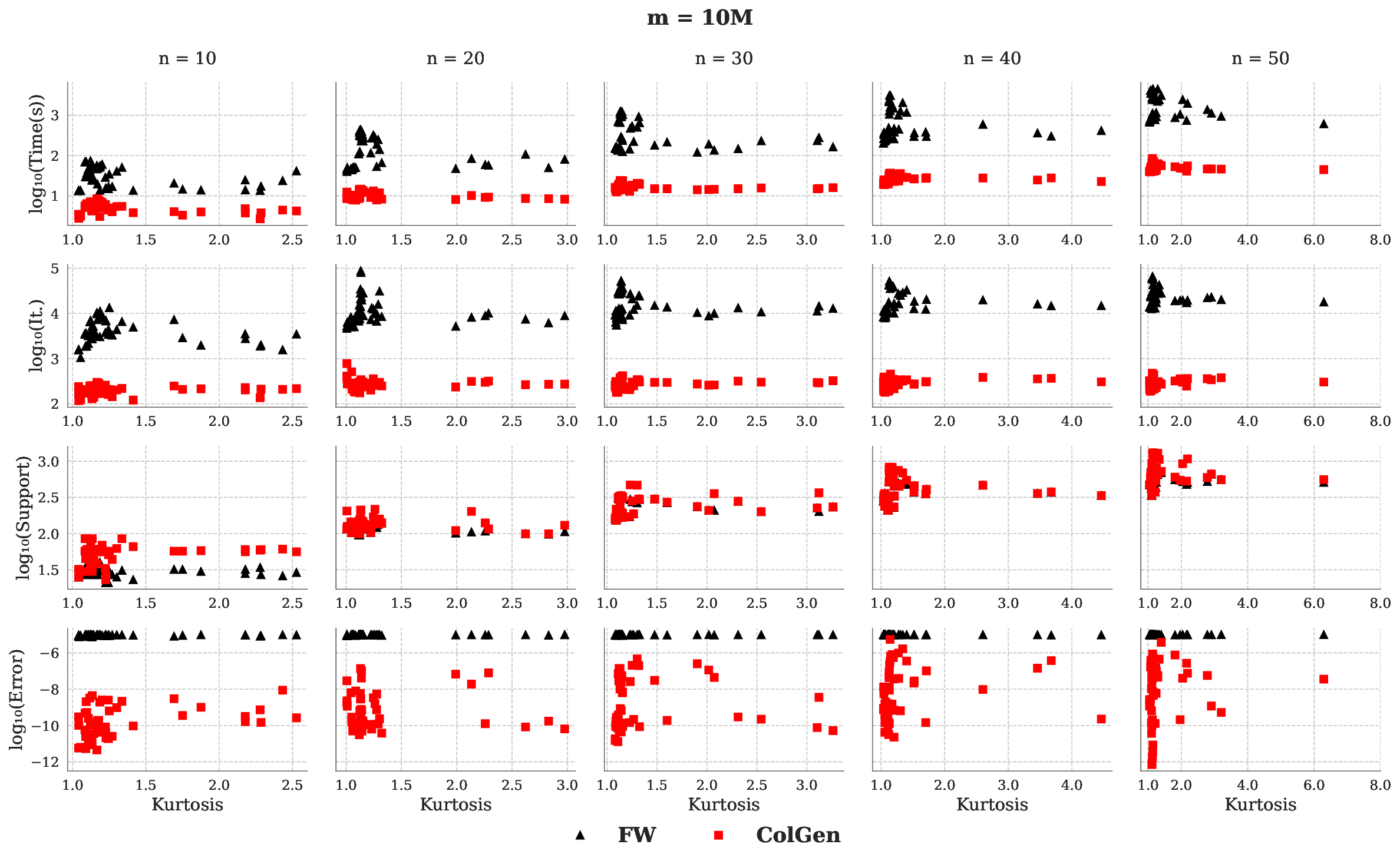}
  \includegraphics[width=\textwidth]{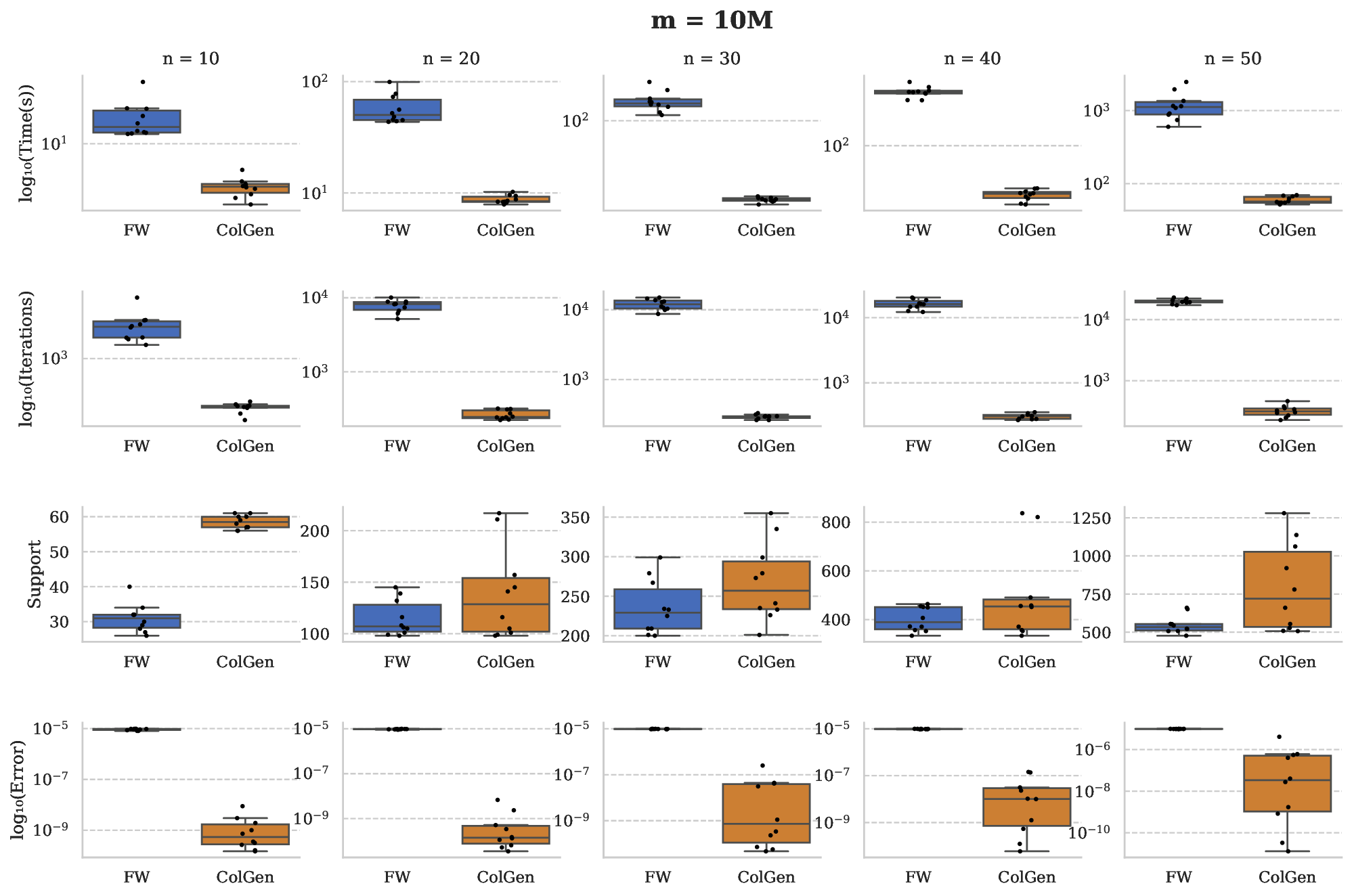}
  \caption{\textbf{Synthetic}: Comparison of FW vs.\ ColGen across metrics and problems. Upper panel: dataset modified using $\sinh–\arcsinh$ transformation, see Section \ref{sec:dataset}. Lower panel: original dataset.}
  \label{fig:fw-cg-comparison}
\end{figure}

\begin{figure}[htbp]
  \centering
  \includegraphics[width=\textwidth]{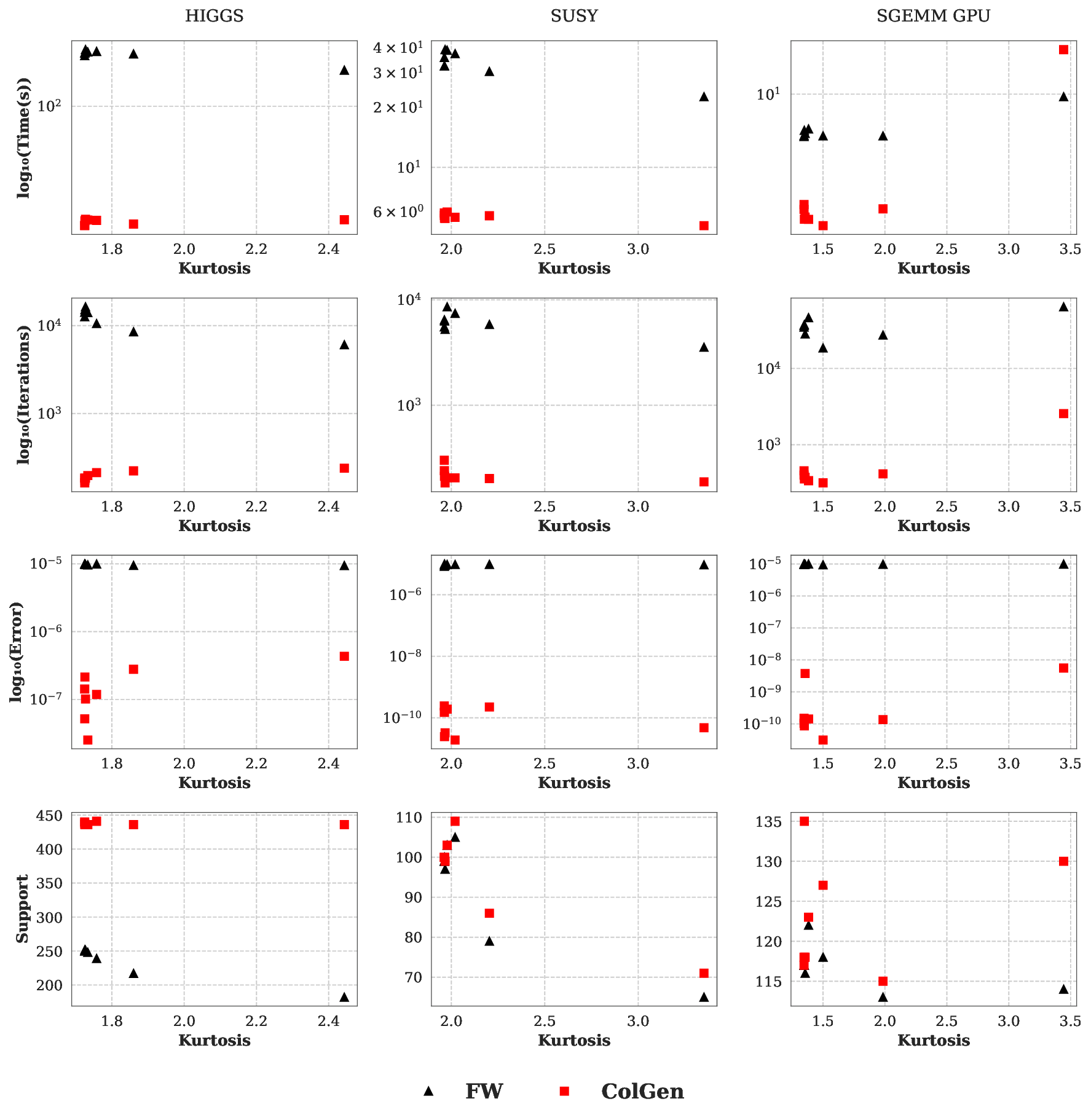}
  \includegraphics[width=\textwidth]{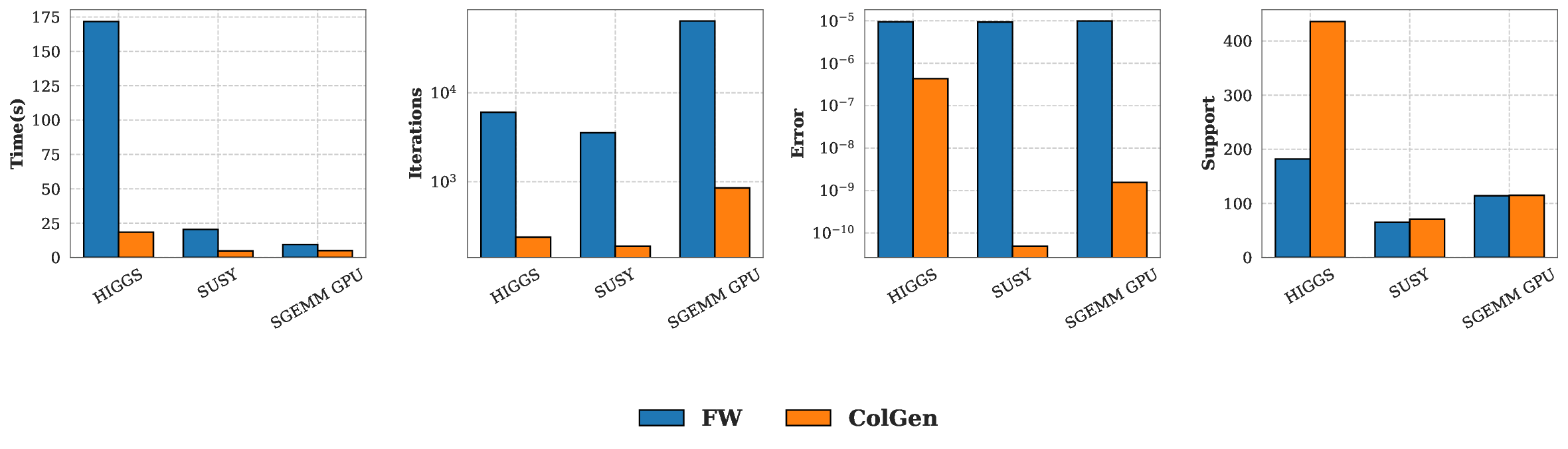}
  \caption{\textbf{Real World}: Comparison of FW vs.\ ColGen across metrics and problems. Upper panel: dataset modified using $\sinh–\arcsinh$ transformation. Lower panel: original dataset.}
  \label{fig:fw-cg-comparison-RW}
\end{figure}

{To conclude this section, we emphasise a critical distinction regarding the computational capabilities demonstrated by Algorithm \ref{alg:column_generation} when compared to the inner primal-dual solver, in this case SDPT3, when used as a stand-alone solver. Indeed, when SDPT3 is invoked directly as a monolithic solver on the full problem formulation \eqref{eq:ED_no_integer_constraints}-\eqref{eq:ED_dual_1}, the computational cost grows rapidly due to the Newton-based nature of Interior-Point Methods, which require the formation and solution of large linear systems per iteration. Consequently, the scale of instances that can be addressed through direct application of SDPT3 remains substantially smaller -- by several orders of magnitude -- than those considered in our experimental study. The computational tractability achieved in our approach stems fundamentally from the Column Generation framework presented in Algorithm \ref{alg:column_generation}, which decomposes the original large-scale problem into a sequence of substantially smaller restricted master problems that remain within SDPT3's computational reach. This observation is consistent with a more general computational experience 
with the use of interior point based Primal-Dual column generation technique \cite{MR2980569, MR3463543}.
To explicitly demonstrate this advantage, we present in Figure~\ref{fig:SDPT3-cg-comparison} a direct comparison between: (i) SDPT3 applied monolithically to problems of increasing dimension $m \in \{10n, 40n, 120n \}$, and (ii) SDPT3 integrated within our Column Generation framework on the same problem instances. The first row of such a figure showcases indeed how the average computational time needed by SDPT3 to solve the instances rapidly increases and, already for relatively small $m$, the Column Generation approach is one order of magnitude faster, while maintaining similar solution quality/properties (see the last two rows in Figure~\ref{fig:SDPT3-cg-comparison}). 
Moreover, comparing the first two rows reveals that, although the standalone application of SDPT3 generally requires fewer IPM iterations, it is consistently slower than the Column Generation approach, as assembling and solving the Newton systems imposes a prohibitive computational footprint. Hence, these results clearly illustrate how, while standalone SDPT3 encounters computational limitations beyond moderate problem sizes, the Column Generation approach extends the solver's applicability to datasets with millions of points. This decomposition strategy represents a core methodological contribution of our work: enabling high-accuracy Interior-Point-based solutions at scales that would otherwise be computationally intractable for direct SDP formulations of the MVEE problem.}

\begin{figure}[htbp]
  \centering
  \includegraphics[width=\textwidth]{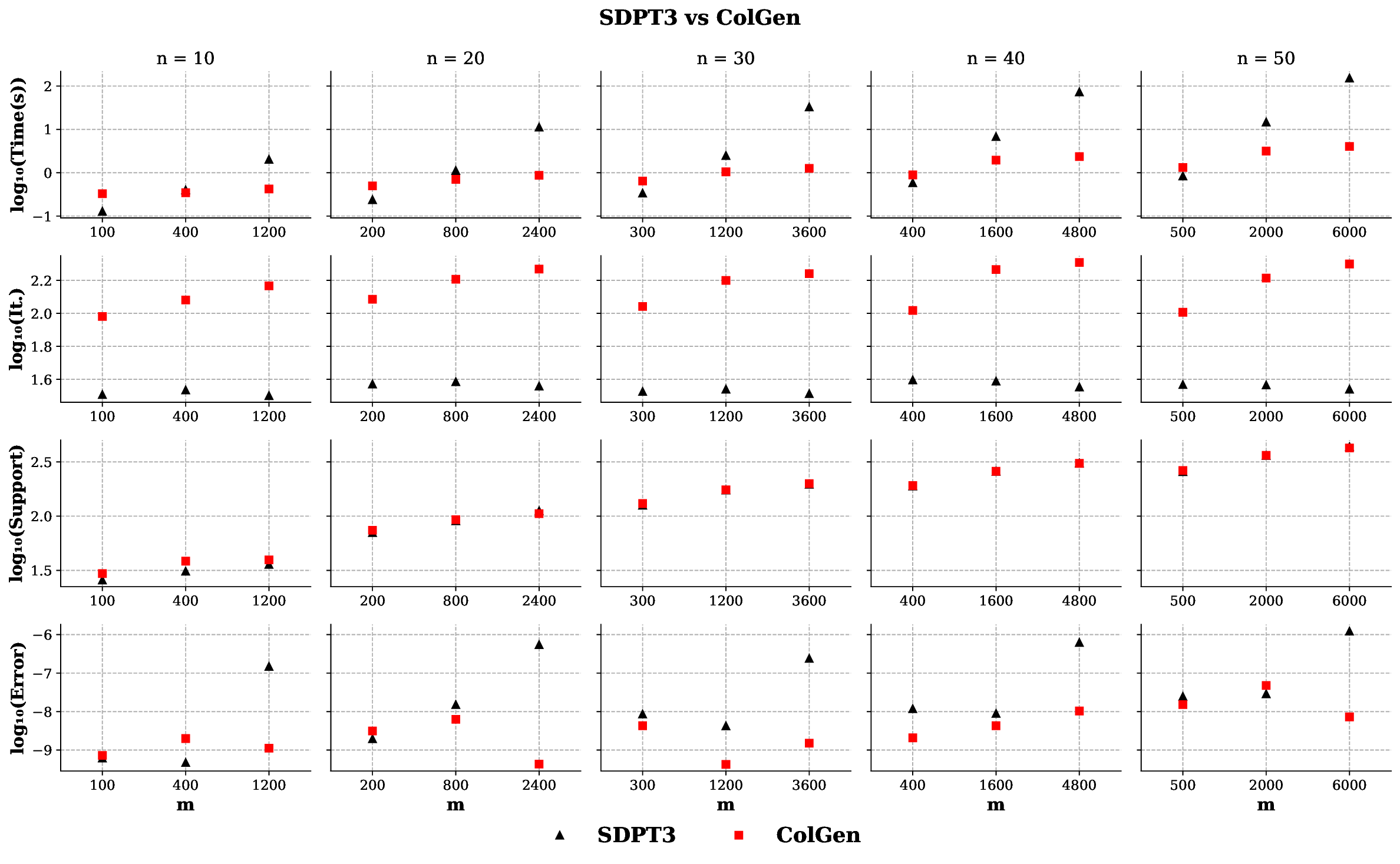}
  \caption{\textbf{Synthetic}: Comparison of SDPT3 vs.\ ColGen across metrics and problems.}
  \label{fig:SDPT3-cg-comparison}
\end{figure}

\subsection{Exact D-optimal Design} \label{sec:exact-D_optimal-experiments}

For the purposes of numerical experiments presented in this section, we consider Algorithm \ref{alg:best-local-search-d-design}, a refined and computationally more demanding -- yet parallelizable -- variant of Algorithm \ref{alg:local-search-d-design}. This version performs an exhaustive search over all possible single exchanges, selecting the one that yields the greatest improvement. The theoretical results developed in Section \ref{sec:bounds_local_search}, {in particular Corollary \ref{cor:worst_case estimate}, remain true also for}  this refined algorithm, as it shares the same termination condition {as Algorithm \ref{alg:local-search-d-design}.} The general algorithmic framework {proposed here} is represented by Algorithm \ref{alg:best-local-search-d-design} where the input set $S$ {can be generated by any method as long as the initial solution is in the domain of the objective function. We refer, interchangeably, as \textit{our proposal} or \textit{ColGen-LS}, to the variant where the input set is generated using Algorithm \ref{alg:column_generation} to obtain a solution $u^*$  of the MVEE and $I$ is initialized using the largest $N$ non-zero elements of $u^*$. The naming \textit{ColGen-LS} is indeed chosen to emphasize that the output is obtained with a combination of column generation and local search algorithms. For the sake of fairness, when we report the total computational time of the ColGen-LS, we always include the run time of Algorithm \ref{alg:column_generation} to find the input set.}

\begin{algorithm}
\caption{Best-improvement local search algorithm for $D$-design}
\label{alg:best-local-search-d-design}
\begin{algorithmic}[1]
\State \textbf{Input:} $S$, $I$ any multi-subset of $S$ of size $N$ such that
   $G = \sum_{i \in I} x_i\,x_i^T$ is a non-singular matrix.
\While{there exist $i \in I$ and $j \in S$ such that 
\[
   \det\bigl(G - x_i\,x_i^T + x_j\,x_j^T\bigr) > \det(G) 
\]}
    \State Choose $(i^\ast, j^\ast)$ that maximizes 
    \[
       \det\bigl(G - x_i\,x_i^T + x_j\,x_j^T\bigr) - \det(G)
    \]
    \State $G \gets G - x_{i^\ast}\,x_{i^\ast}^T + x_{j^\ast}\,x_{j^\ast}^T$
    \State $I \gets \bigl(I \setminus \{i^\ast\}\bigr) \cup \{j^\ast\}$
\EndWhile
\State \textbf{Output:} $(G, I)$
\end{algorithmic}
\end{algorithm}



{The purpose of the first series of experiments is to assess the quality of our proposal by comparing both the primal objective function values and computational time against the software from \cite{MR4774635} (denoted in the following by \textit{Boscia}) on medium-sized problem instances. We acknowledge, again, that the computational time comparison should be interpreted with appropriate caveats: Boscia is an exact solver that provides globally optimal solutions, whereas our proposal is a heuristic method with worst-case approximation guarantees. Nevertheless, for the sake of completeness, we report both metrics to demonstrate the practical performance of our approach.}
For Boscia, we use the default parameters. In  Figures  \ref{fig:boscia-cg-comparison1020} and \ref{fig:boscia-cg-comparison304050}, we report the results of numerical experiments obtained on synthetic problems created using the generator provided in the solver Boscia using the option \texttt{``correlated''} when $(n,m) \in \{10,20,30,40,50\} \times \{ 100,200,300,400,500 \}$ and when for every $(m,n)$ $10$ instances are generated. Broadly speaking, such a generator produces instances by sampling one Gaussian distribution with a randomly generated mean and covariance matrix. As highlighted by the upper panels of the aforementioned figures and as already observed in Section \ref{sec:prob_formulation}, the computational time for the solver Boscia (in black), in the case $N=n$, roughly increases by two orders of magnitude when $m$ goes from $100$ to $500$. It is also interesting to note how such a trend still remains true when $N>n$ and $n \in \{30,40,50\}$, suggesting that the problems become more and more difficult for the solver when the possible size of the support increases. For $n=40,50$ we capped the maximum allowed time to $1800$ seconds, and such a wall-time is usually hit by Boscia before producing a solution satisfying the required accuracy. On the other hand, our proposal (in red) is able to solve all the instances in less than $10$ seconds. {Notably}, in general, our proposal produces better or comparable objective functions, especially when $N=n$, see the lower panels of Figures \ref{fig:boscia-cg-comparison1020} and \ref{fig:boscia-cg-comparison304050}. {Hence, we explicitly note how such results systematically show that the } worst-case complexity reported in Corollary \ref{cor:worst_case estimate} is an over-pessimistic bound as the function values obtained by {ColGen-LS} are comparable with the function values computed using Boscia, which is, theoretically, an enumerative method and should return the \textit{true} optimizer. 

\begin{figure}[htbp]
  \centering
  \includegraphics[width=0.7\textwidth]{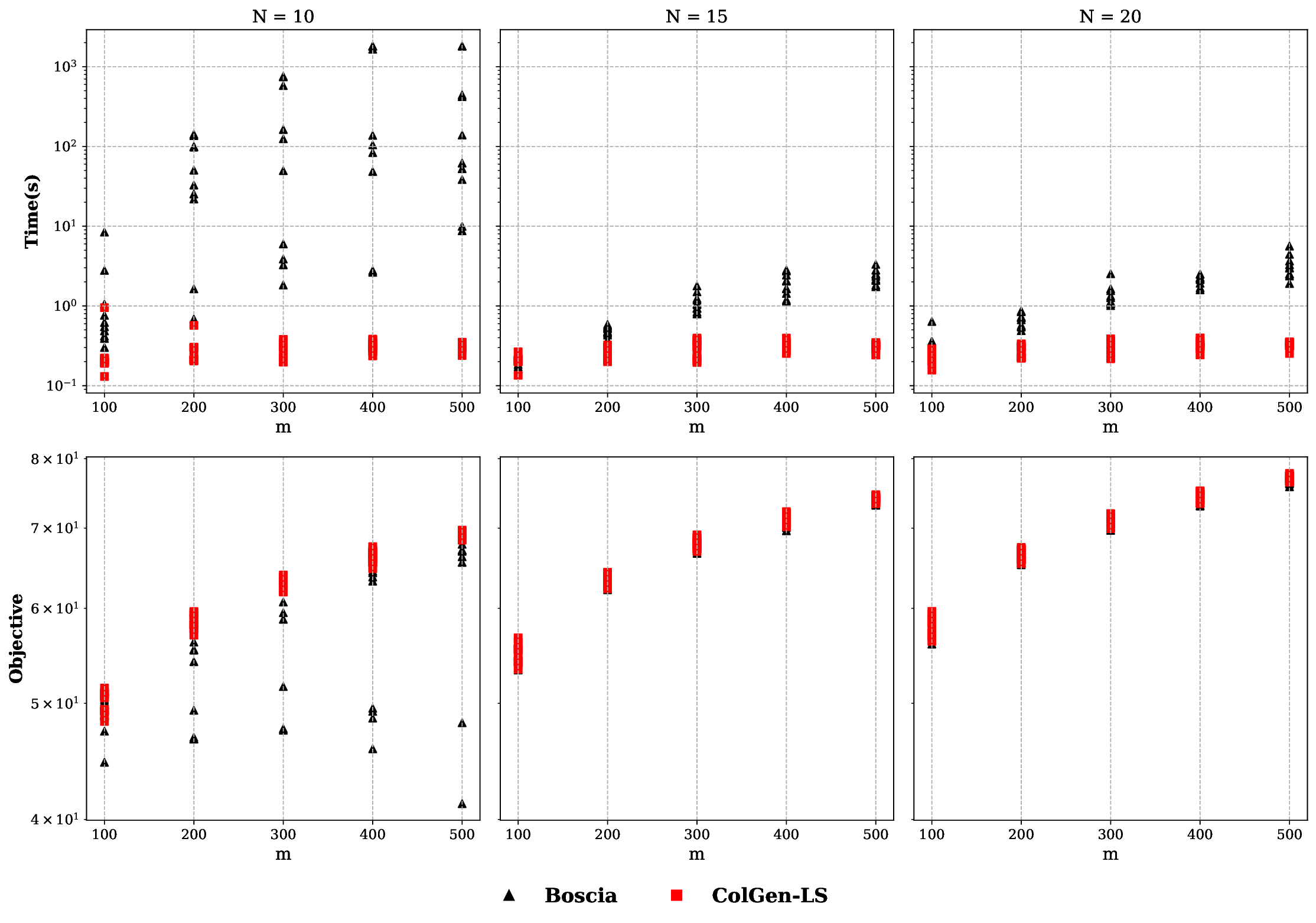}
  \includegraphics[width=0.7\textwidth]{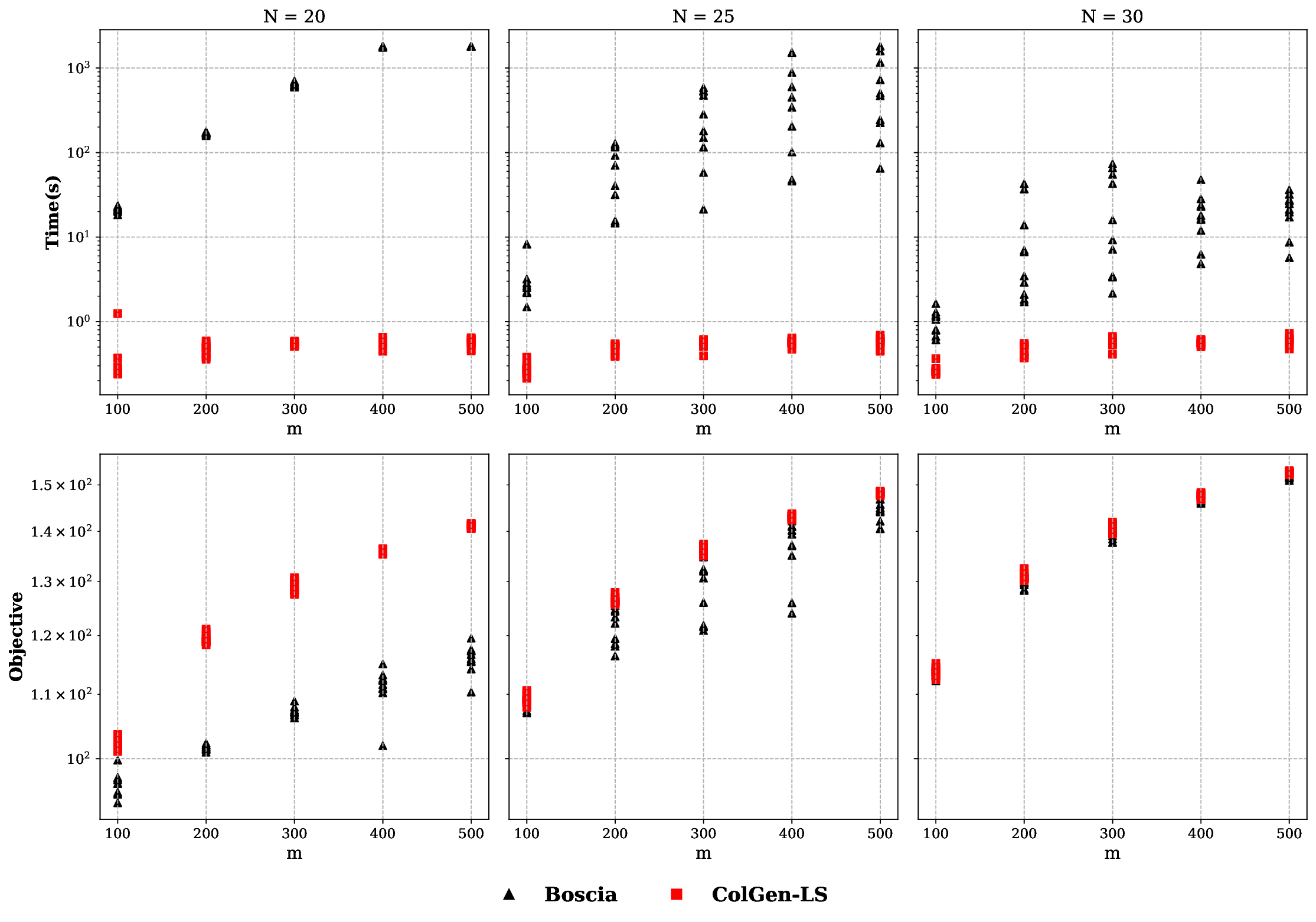}
  \caption{Top to bottom: $n=10,20$. \textbf{Synthetic} problems using generator from \cite{MR4774635} with \texttt{``correlated''} option.}
  \label{fig:boscia-cg-comparison1020}
\end{figure}



\begin{figure}[htbp]
  \centering
  \includegraphics[width=0.7\textwidth]{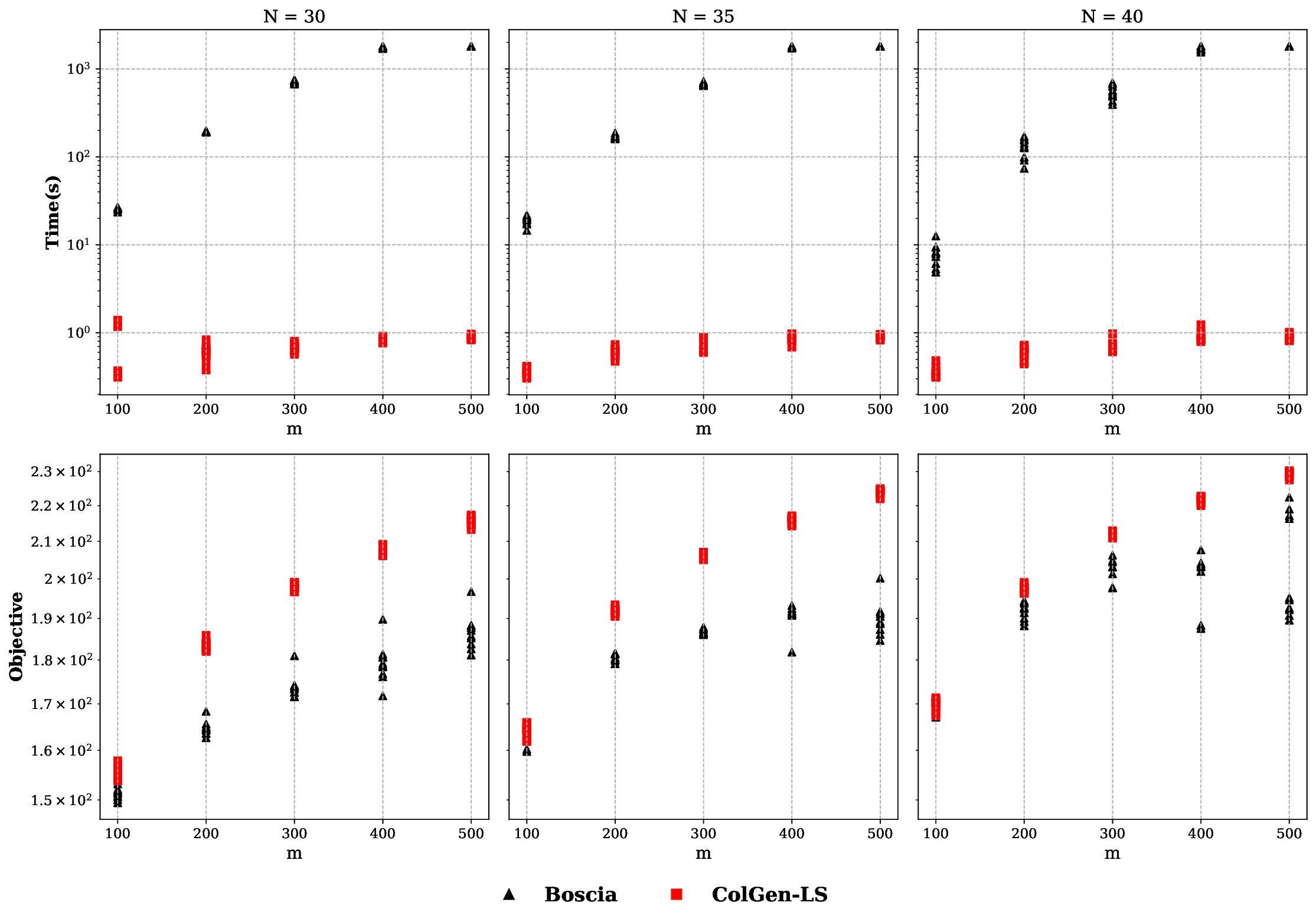}
  \includegraphics[width=0.7\textwidth]{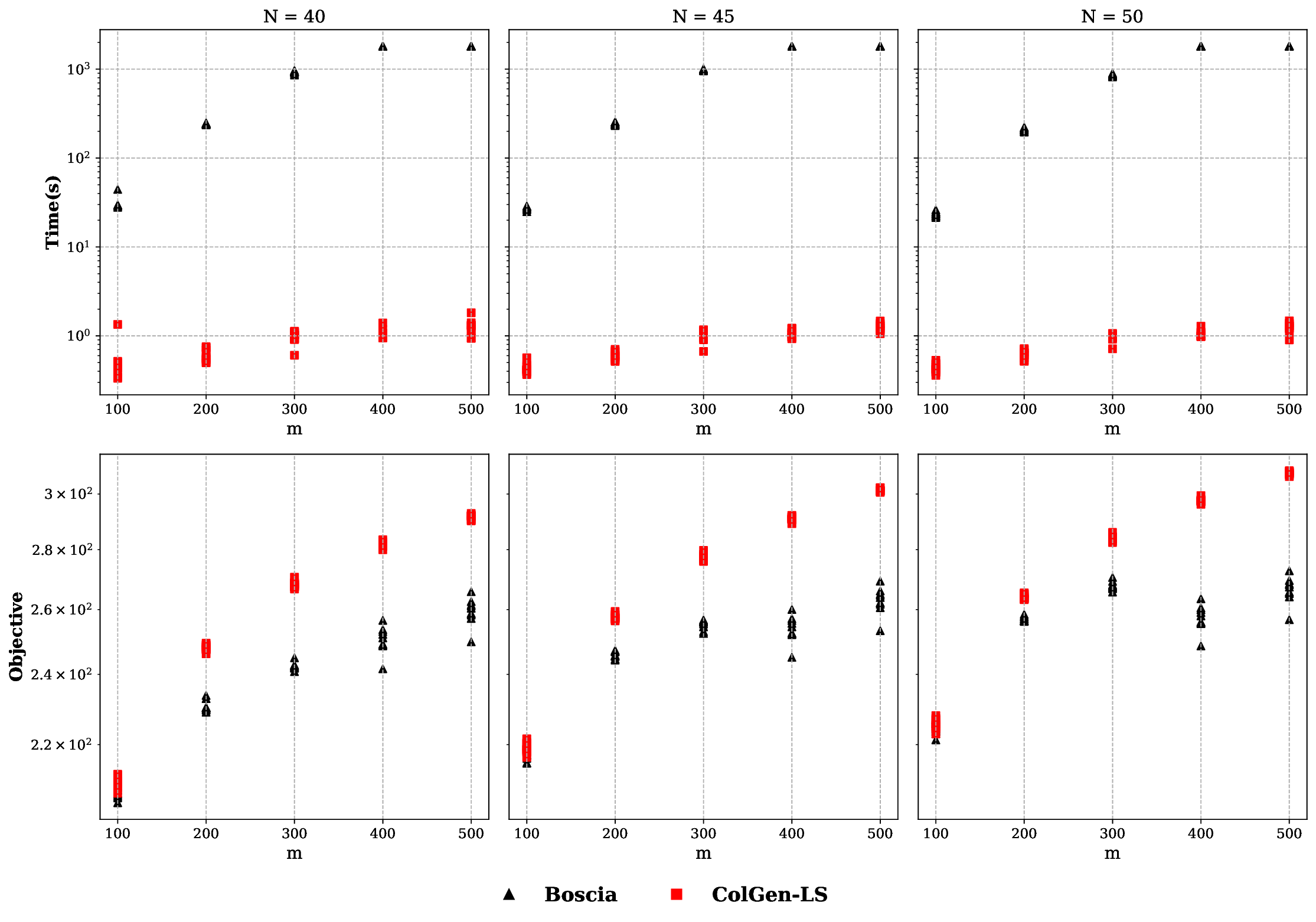}
  \includegraphics[width=0.7\textwidth]{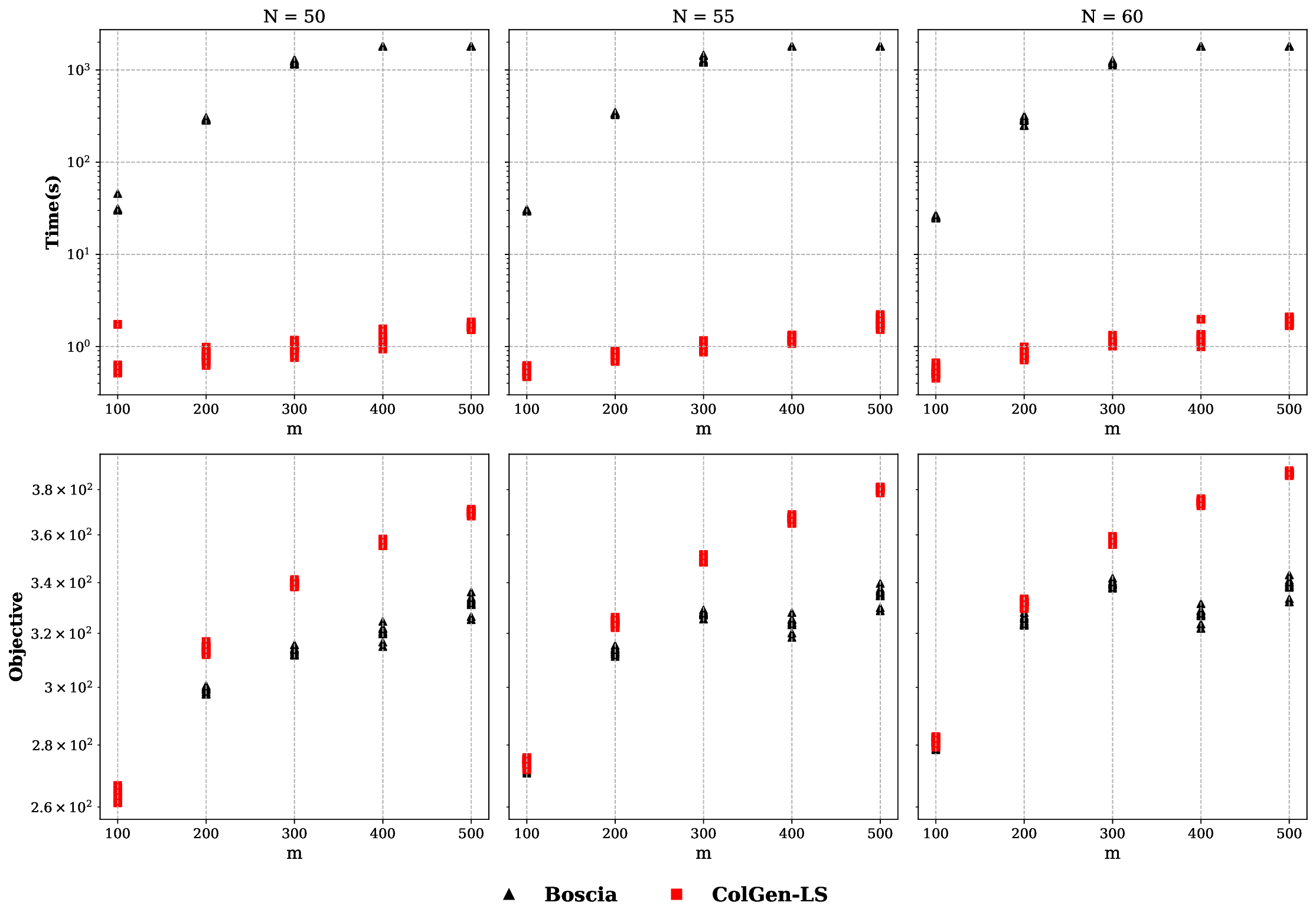}
  \caption{Top to bottom: $n=30,40,50$. \textbf{Synthetic} problems using generator from \cite{MR4774635} with \texttt{``correlated''} option.}
  \label{fig:boscia-cg-comparison304050}
\end{figure}


In the second experiment, see Figures  \ref{fig:locsearch-FW_cg-comparison1M} and \ref{fig:cg-RW-Kurtosis}, we compare FW and {ColGen} when used to generate the initial support $S$ in Algorithm \ref{alg:best-local-search-d-design} on \textbf{Synthetic} and \textbf{Real World} datasets. {We refer to these variants as FW-LS and ColGen-LS, respectively, as above.}
In particular, in order to offer an alternative measure for evaluating the quality  of the computed integer solution, we will use the following \textit{gap} measure, an analogue of the \textit{relative} MIP gap for Mixed Integer Programming, i.e., given $G$ output of Algorithm \ref{alg:best-local-search-d-design} and $u^*$  solution of \eqref{eq:ED_no_integer_constraints}, we define 
\begin{equation*}
  gap:= \frac{\log(\det(XU^*X^T))-\log(\det(G))}{|\log(\det(XU^*X^T))|} .
\end{equation*}
Given ${u_{E}^*}$ solution of \eqref{eq:ED}, it is easy to see that $gap \geq 0$  and that
\begin{equation*}
gap \geq \frac{\log(\det(XU^*X^T))-\log(\det(XN{U_{E}^*}X^T))}{|\log(\det(XU^*X^T))|},   
\end{equation*}
where the right-hand side is the best achievable \textit{relative} gap for an integer solution.

The figures present the computational time (upper panel) and the optimality gap (lower panel) relative to the Column Generation {(ColGen-LS) and the Frank–Wolfe (FW-LS)} methods. As shown by these results, our proposed approach solves all instances within 150 seconds, whereas {FW-LS} remains consistently slower. Regarding the quality of the integer solutions, both frameworks typically exhibit a gap of the order of \texttt{1e-1}. Intriguingly, although FW generally identifies sparser supports than ColGen (cf. the discussion in Section \ref{sec:numerics_MVEE}), it still produces integer solutions whose gap is comparable to that delivered by ColGen. This finding is somewhat counterintuitive, since one would expect that applying the local search Algorithm \ref{alg:best-local-search-d-design} on a larger set $S$ would, in general, yield superior quality solutions. On the other hand, this is perfectly in line with the theory developed in Section \ref{sec:local_search}, basically postulating that the worst case bounds only depend on the objective value.

\begin{figure}[htbp]
  \centering
  \includegraphics[width=1.2\textwidth]{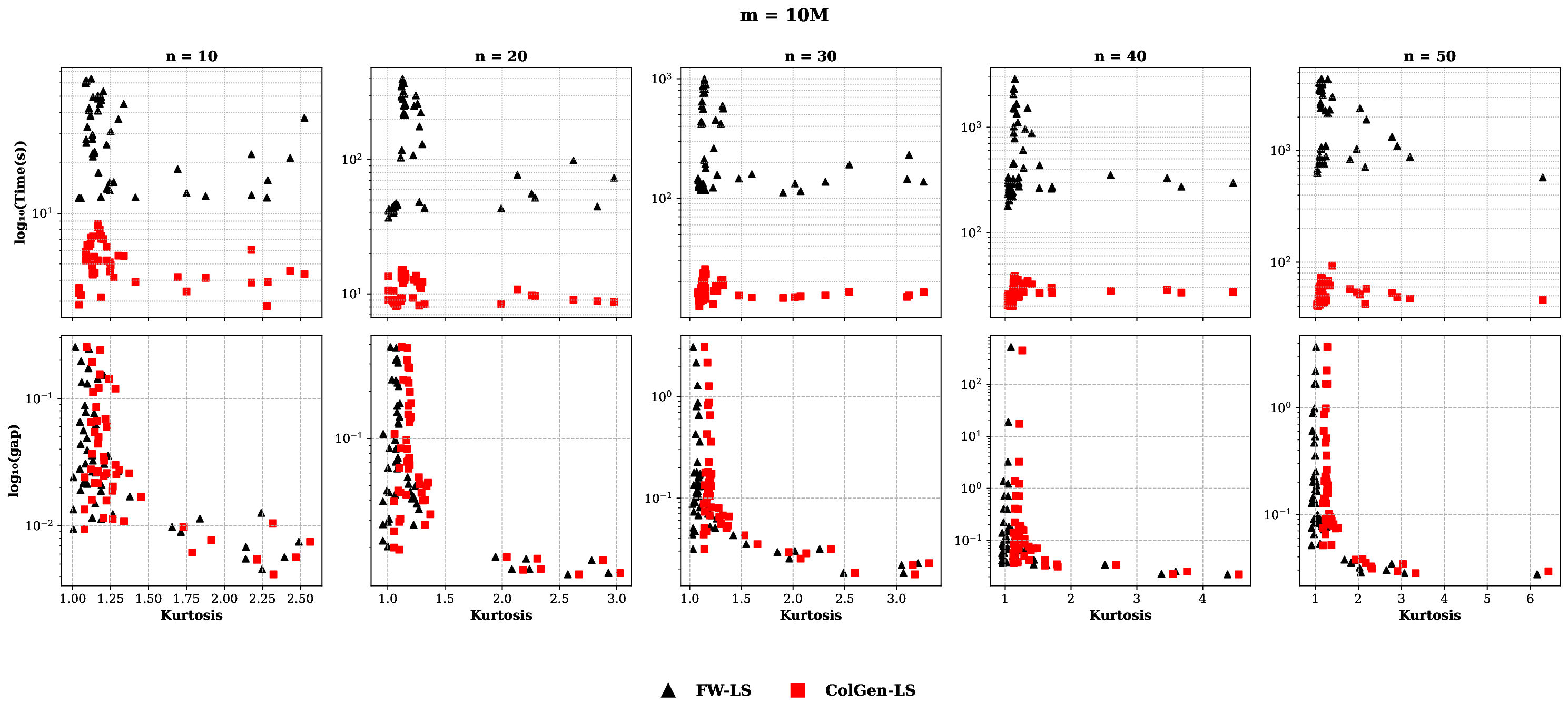}
  \includegraphics[width=1.18\textwidth]{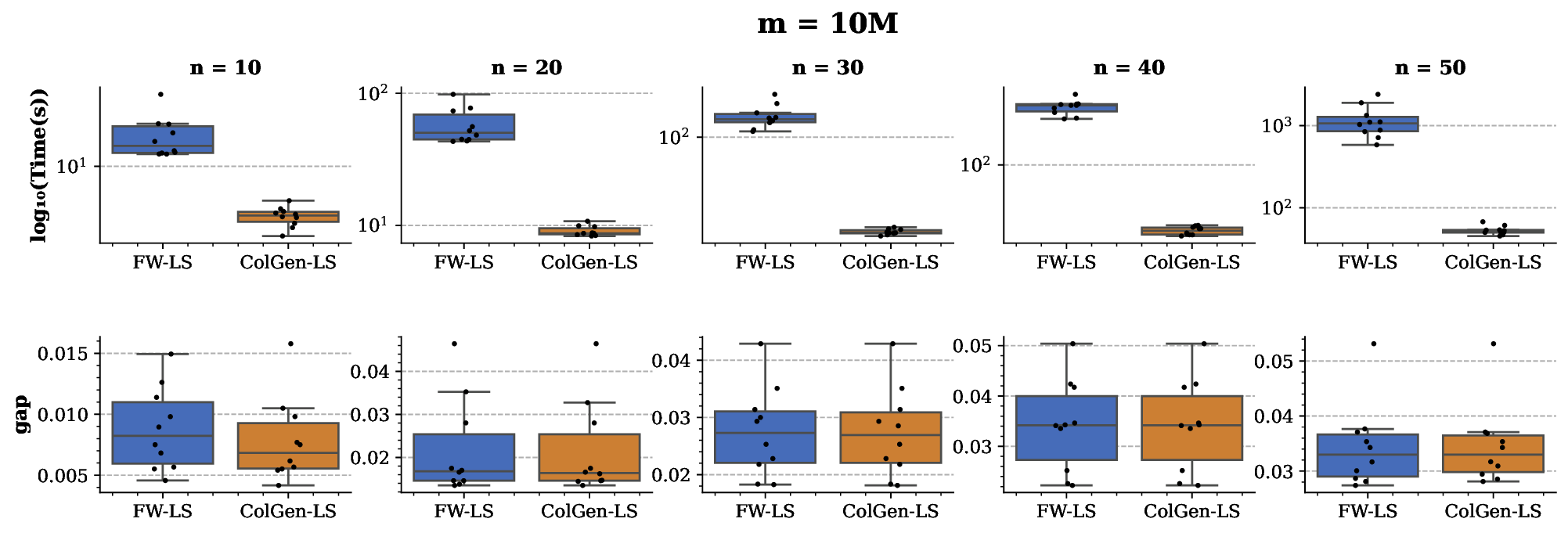}
  \caption{\textbf{Synthetic} dataset $m=10M$. Running time and gap values of Algorithm \ref{alg:best-local-search-d-design} with Algorithm \ref{alg:column_generation} (in red) and FW (in black). Upper panel: $\sinh-\arcsinh$ transformation. Lower Panel: \textit{original dataset}.}
  \label{fig:locsearch-FW_cg-comparison1M}
\end{figure}

\begin{figure}[htbp]
  \centering
  \includegraphics[width=\textwidth]{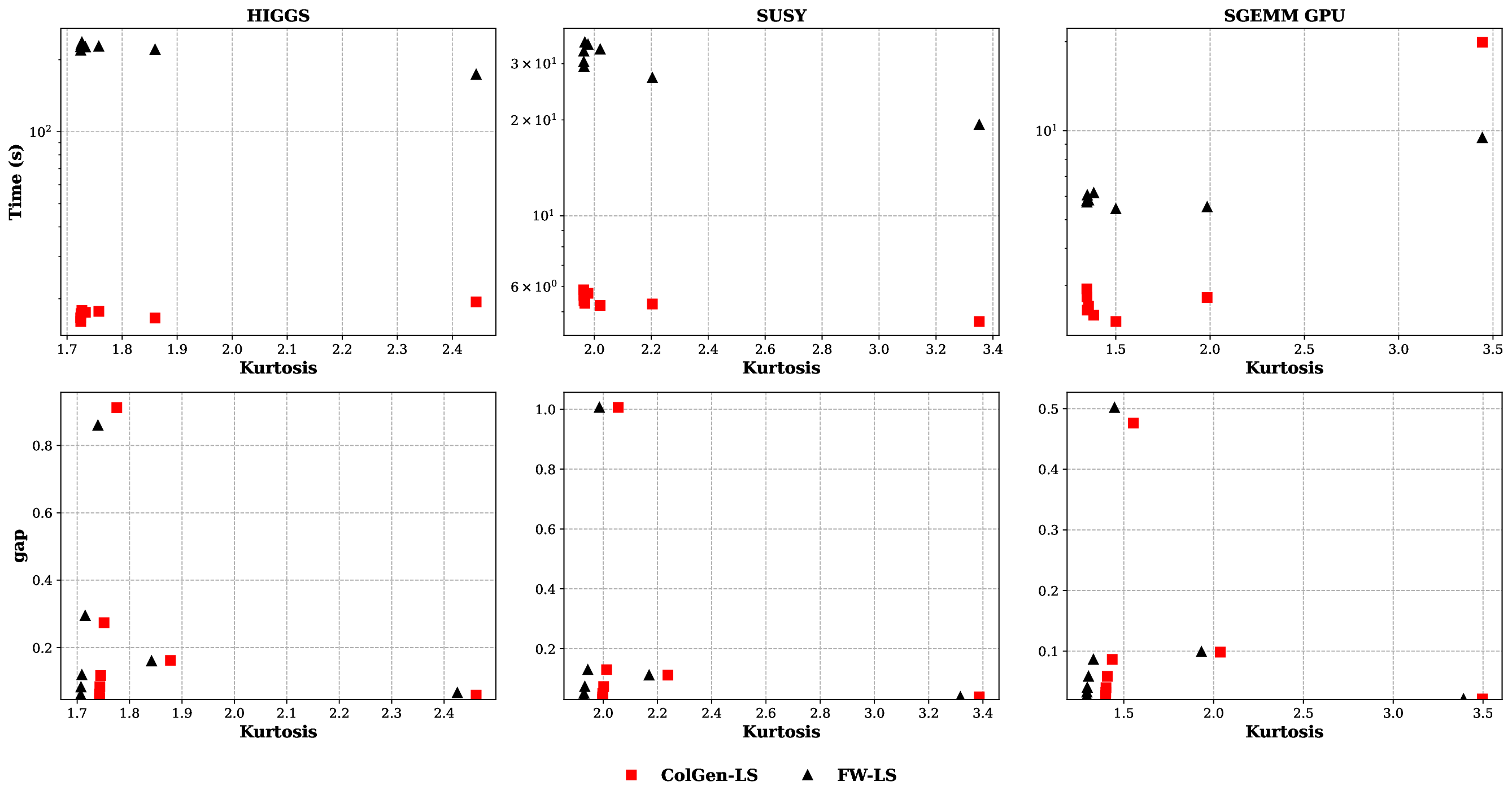}
  \caption{\textbf{Real World}: Running time and gap for FW and ColGen for $\sinh-\arcsinh$ transformed datase.}
  \label{fig:cg-RW-Kurtosis}
\end{figure}

{Moreover, {to assess the performance} of our approach w.r.t a \textit{purely greedy approach}}, we compare our proposal with Algorithm \ref{alg:best-local-search-d-design} when the input $S$ is the full dataset, i.e., without the preliminary support identification performed using Algorithm \ref{alg:column_generation}. {
In particular, to ensure a fair comparison, only in this experiment, the initialisation of the multi-set $I$ is handled as follows. When Algorithm \ref{alg:best-local-search-d-design} is applied with preliminary support identification via Algorithm \ref{alg:column_generation}, the starting point is chosen as a random multi-set drawn from the identified support (referred to as \texttt{ColGen-LS-RInit} in what follows). When no preliminary support identification is performed, $I$ is instead initialised as a random multi-set drawn from the full dataset (referred to as \texttt{LocalSearch-RInit} in what follows). For this experiment,} we restrict our analysis to the case $N=n$, as the experiments conducted and presented so far have shown that this regime includes the most challenging cases for the computation of exact D-optimal design solutions. {In Figure \ref{fig:locsearch-cg-comparison} we report the total computational time for five random initializations of $I$ for \texttt{ColGen-LS-RInit} and \texttt{LocalSearch-RInit} and the best obtained objective value. It is important to note that the reported computational time for \texttt{ColGen-LS-RInit}  includes the computational time needed by Algorithm \ref{alg:column_generation} for the identification of the support.} As highlighted in Figure \ref{fig:locsearch-cg-comparison}, \texttt{{ColGen-LS-RInit} outperforms \texttt{LocalSearch-RInit}} by two orders of magnitude in terms of computational time for the \textbf{Synthetic} dataset with $m=1M$ (cf., first and second rows subplots) while obtaining equal or better objective function values (cf., third and fourth rows subplots).
Moreover, the comparison of the reported computational times for $m=100K$ and $m=1M$ highlights that {also in this case, ColGen-LS-RInit} exhibits a limited dependence on $m$.
\begin{figure}[htbp]
  \centering
  \includegraphics[width=1.2\textwidth]{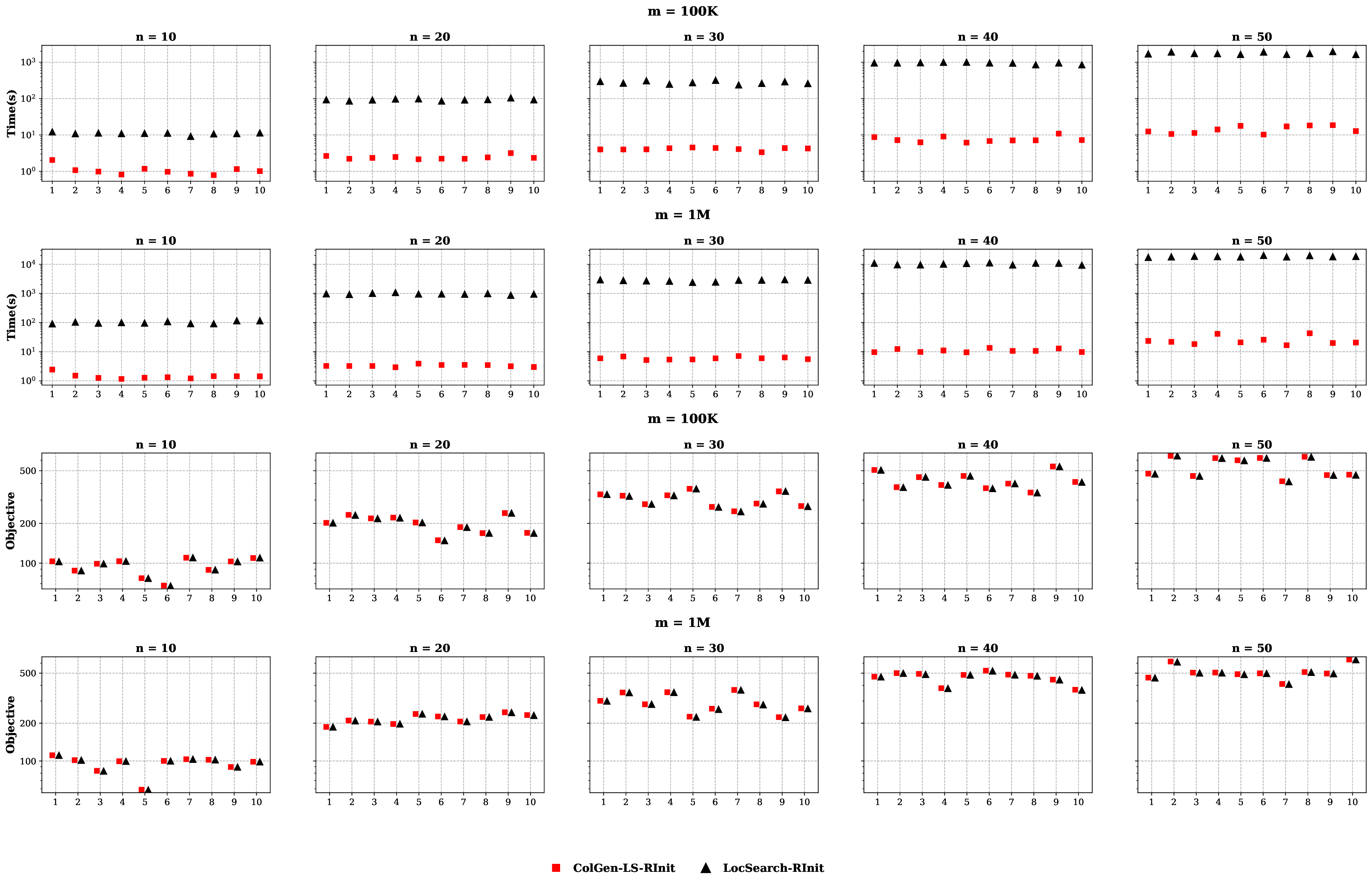}
  \caption{\textbf{Synthetic}: {case $N=n$. \texttt{ColGen-LS-RInit} uses a random multi-set drawn from the support identified by Algorithm \ref{alg:column_generation}, while \texttt{LocalSearch-RInit} uses a random multi-set drawn from the full dataset. Total running time and best objective function for five random initialization obtained using Matlab's function \texttt{randperm}.} }
  \label{fig:locsearch-cg-comparison}
\end{figure}

{
Finally, we compare ColGen-LS with methods specifically designed for high-dimensional problems that leverage statistical principles. Figure \ref{fig:locsearch_Iboss-comparison1M} presents a comparative analysis of our proposal with \texttt{IBOSS} \cite[Alg. 1]{MR3941263}. This comparison is particularly pertinent in this context, as IBOSS implements D-optimal-driven subsampling \cite{MR3941263} and has demonstrated superior performance relative to classical subsampling methods w.r.t. the \textit{statistical properties} of the resulting solution. To ensure a fair comparison, we employ computational time and solution $gap$ as our evaluation metrics. We note explicitly, moreover, that IBOSS imposes a structural constraint whereby the minimum experiment size is $N=2n$. The numerical results presented in Figure \ref{fig:locsearch_Iboss-comparison1M} correspond to this choice of $N$, as larger values would yield insufficiently challenging test problems, as discussed extensively in Remark \ref{rem:difficulty} and corroborated by Corollary \ref{cor:worst_case estimate}. The results summarized in Figure \ref{fig:locsearch_Iboss-comparison1M} demonstrate that while IBOSS consistently achieves the best computational time, it produces solutions with optimality gaps that are roughly one order of magnitude larger than those obtained by ColGen.} 

\begin{figure}[htbp]
  \centering
  \includegraphics[width=1.2\textwidth]{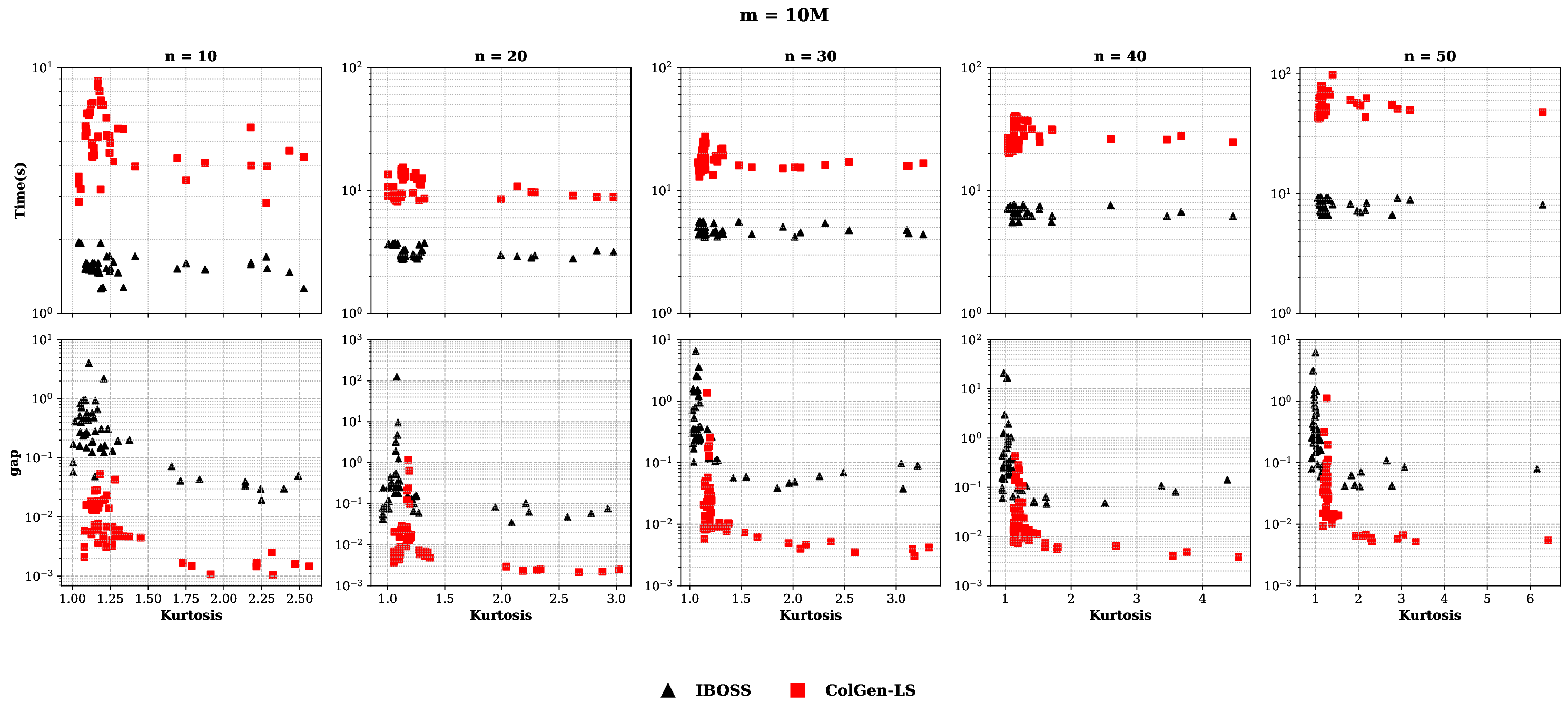}
  \includegraphics[width=1.18\textwidth]{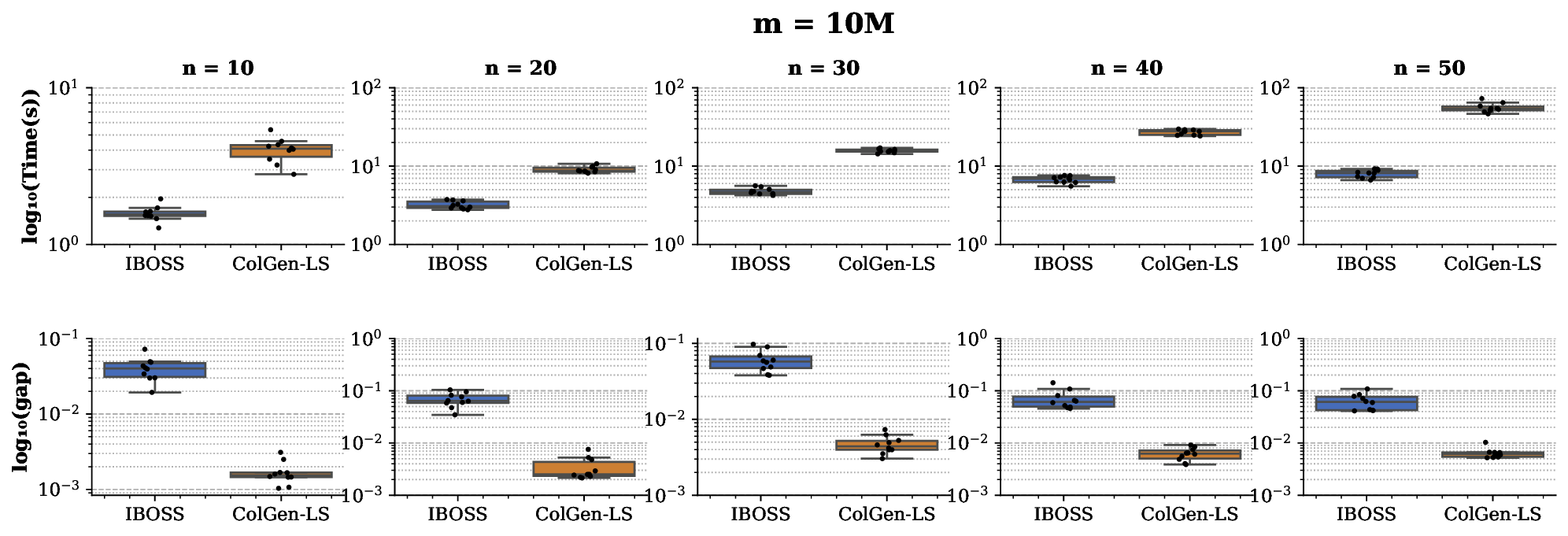}
  \caption{$N=2n$. \textbf{Synthetic} dataset $m=10M$. Running time and gap values of Algorithm \ref{alg:best-local-search-d-design} with Algorithm \ref{alg:column_generation} (in red) and IBOSS (in black). Upper panel: $\sinh-\arcsinh$ transformation. Lower Panel: \textit{original dataset}.}
  \label{fig:locsearch_Iboss-comparison1M}
\end{figure}

\section{Conclusions \textcolor{red}{and future work}}

In this work, we have addressed the longstanding challenge of computing exact D-optimal experimental designs in the regime where the number of candidate points far exceeds the dimensionality of the regression model and the total number of allowed experiments is close to the number of parameters to be estimated, i.e., $N\approx n$. Building on the duality between the D-optimal design problem and the minimum-volume enclosing ellipsoid (MVEE) problem, we have introduced a hybrid column-generation framework that integrates the following key elements:
\begin{itemize}
    \item \textit{Rapid Support Identification via Column Generation and Interior-Point SDP}. This is achieved by formulating the continuous relaxation, the \textit{limit problem}, as a primal–dual pair and by deploying a column-generation strategy -- Algorithm \ref{alg:column_generation} -- to identify, in an iterative fashion, the small support set that carries all of the mass of the continuous optimum.  At each iteration, we solve a restricted master problem over a modest-sized subset of candidate points using a high-accuracy, primal–dual Interior-Point SDP solver (SDPT3). 

\item \textit{Restricted Local Search with Provable Bounds}. Once the support $S$ of the \textit{limit problem} has been determined, we invoke a greedy local‐search algorithm -- Algorithm \ref{alg:local-search-d-design} -- that operates using only points of $S$.  By leveraging the bound provided by the optimiser of the \textit{limit problem}, we prove that restricting the local search to $S$ yields an exact design whose worst‐case approximation factor to the optimum is identical to that obtained when running the same local‐search procedure on the entire dataset.  Crucially, this means that no theoretical guarantee is sacrificed: performing a local search restricted to the support of the \textit{limit problem}  yields the same worst-case bound in $N$ and $n$ as in recently proposed analyses, but at a fraction of the computational cost.

\item \textit{Extensive Numerical Validation on Synthetic and Real-World Data}.
   We benchmark our proposal for the solution of MVEE against a state-of-the-art Frank–Wolfe‐type algorithm on very large synthetic datasets and on large UCI real-world datasets (up to $m=10^7$).  Across all instances, including those modified to be more challenging, our proposal identifies the continuous support tens to hundreds of times faster than FW, attains smaller duality gaps, and exhibits far fewer iterations, thereby demonstrating both superior convergence and robustness.  Moreover, we compare our full pipeline -- support identification followed by restricted local search -- against a state-of-the-art nonlinear-mixed-integer solver -- Boscia -- on medium-scale synthetic instances (up to $m=500$, $n=50$).  Even for the most challenging cases, i.e., when $N=n$, our method consistently produces exact designs of equal or better objective value in under $30$ seconds, whereas Boscia often fails to reach optimality within the $1800$ second time limit. Finally, we evaluated the impact of preliminary support identification on the efficiency of the local‐search procedure by conducting experiments on large‐scale synthetic datasets with $m=10^5$ and $m=10^6$ candidate points. In these experiments, our column‐generation–driven approach delivered up to a two‐order‐of‐magnitude reduction in runtime compared to a naive local‐search algorithm applied directly to the full dataset. Moreover, when applied to the UCI real-world dataset, our method consistently produced exact designs with a negligible mixed‐integer programming (MIP) optimality gap, demonstrating that the restricted local search confined to the support of the limit problem guarantees the quality of the solution while drastically reducing computational overhead.

\end{itemize}

In summary, by combining a column-generation viewpoint with second-order optimization techniques and a tailored local-search, we have demonstrated that exact D-optimal designs for large datasets with small to moderate number of features can be computed efficiently, reliably and accurately, also at scales several orders of magnitude larger than previously possible.  

\textcolor{black}{Looking forward, a natural and promising direction for future work is the extension of the proposed framework to A-optimal experimental design, which is the other popular criterion among Kiefer's. The duality theory underlying our column generation approach extends to A-optimality through the analysis carried out in \cite{MR4307385}, where a dual formulation and the corresponding optimality conditions are derived in a form structurally analogous to the MVEE duality exploited in Section~\ref{sec:ColumnGen}. This dual characterization provides the necessary pricing rule for an A-optimal column generation scheme, while elimination criteria for non-support points -- the A-optimal counterpart of the Harman--Pronzato condition used in Algorithm~\ref{alg:column_generation} -- are available in \cite{MR4193082}. Moreover, the worst-case approximation guarantees obtained in \cite{madan2019combinatorial} for local search algorithms on the full dataset have an A-optimal analogue with comparable degradation in the regime $N \approx n$, suggesting that the restricted local search argument of Section~\ref{sec:local_search} could be used for A-optimality. We chose to focus exclusively on D-optimality in this work for two reasons. First, the D-optimal limit problem admits a particularly clean SDP representation that is well-suited to high-accuracy primal--dual interior-point solvers such as SDPT3, whereas the A-optimal criterion has no native conic form and must be lifted via auxiliary Schur-complement blocks, increasing the computational footprint of each restricted master problem and warranting a dedicated implementation and benchmarking effort. Second, the benchmarking algorithms and characterization of more challenging datasets using kurtosis that we have used in our computational experiments (following the setup in \cite{MR4655115}) are not readily available for A-optimality. We therefore view the systematic adaptation of the proposed framework to A-optimality, together with a thorough numerical comparison against state-of-the-art mixed-integer methods and local search methods, such as those in \cite{MR4774635,madan2019combinatorial}, as a worthwhile and timely future research direction.
}

\section*{Acknowledgments}
The authors acknowledge the use of the IRIDIS High Performance Computing Facility and associated support services at the University of Southampton, in the completion of this work.

\section*{Statements \& Declarations}

\subsection*{Funding}
The authors declare that no funds, grants, or other support were received during the preparation of this manuscript.

\subsection*{Competing Interests}
The authors have no relevant financial or non-financial interests to disclose.

\subsection*{Data Availability}
The generators for the synthetic datasets are available for review and will be published upon publication at \url{https://github.com/StefanoCipolla/D_Optimal_Design_Matlab}. In addition, the Real World datasets are available for download from \url{https://archive.ics.uci.edu/datasets}.

\subsection*{Code Availability}
The full code is available for review. We remark that a set of packages were used in this study, that were either open source or available for academic use. Specific references are included in this published article.

\bibliography{sn-bibliography}

\end{document}